\documentclass[11pt,leqno,twoside]{article}

\usepackage{geometry}
\usepackage{amsmath}
\usepackage{amsfonts}
\usepackage{amsthm}
\usepackage{indentfirst}
\usepackage{float}

\numberwithin{equation}{section}
\theoremstyle{plain}
\newtheorem{theorem}[equation]{\indent\rm T\,h\,e\,o\,r\,e\,m\;}

\newtheorem{proposition}[equation]{\indent\rm P\,r\,o\,p\,o\,s\,i\,t\,i\,o\,n\;}

\theoremstyle{definition}
\newtheorem{definition}[equation]{\indent\rm D\,e\,f\,i\,n\,i\,t\,i\,o\,n\;}

\theoremstyle{remark}
\newtheorem*{remark}{\indent\rm R\,e\,m\,a\,r\,k\;}

\renewenvironment{proof}{\indent\rm P\,r\,o\,o\,f.\;}{\hfill $\square$ \\ \indent}

\makeatletter                                                  %
\renewcommand*{\@seccntformat}[1]{
  \csname the#1\endcsname\;-                                   %
}                                                              %
\renewcommand{\section}{\@startsection{section}{1}{0mm}        %
   {1.5\baselineskip}
   {1\baselineskip}
   {\indent\normalfont\normalsize\bfseries}
   }                                                           %
\renewcommand*{\@seccntformat}[1]{
  \normalfont\bfseries\csname the#1\endcsname\;-               %
}                                                              %
\renewcommand\subsection{\@startsection                        %
  {subsection}{2}{0mm}
  {1.5\baselineskip}
  {1\baselineskip}
  {\indent\normalfont\normalsize\itshape}}
\renewcommand*{\@seccntformat}[1]{
  \normalfont\bfseries\csname the#1\endcsname\;-               %
}                                                              %
\renewcommand\subsubsection{\@startsection                     %
  {subsubsection}{2}{0mm}
  {1.5\baselineskip}
  {1\baselineskip}
  {\indent\normalfont\normalsize\texttt}}
\makeatother                                                   %
\usepackage{amsmath,amsthm,amscd,color}
\usepackage{amssymb}
\usepackage{amsfonts}
\usepackage{latexsym}
\usepackage{mathrsfs}
\usepackage{bbding,array,tabu}
\usepackage{graphicx}
\usepackage{mathscinet}
\usepackage{mathdots}
\usepackage{hyperref}
\usepackage{subfigure}
\usepackage[normalem]{ulem}
\usepackage{color}
\usepackage{xcolor}
\usepackage{mathtools}
\newsavebox{\mybox}
\usepackage[T1]{fontenc}
\usepackage{lscape}

\usepackage{array}
\newcolumntype{L}[1]{>{\raggedright\let\newline\\\arraybackslash\hspace{0pt}}m{#1}}
\newcolumntype{C}[1]{>{\centering\let\newline\\\arraybackslash\hspace{0pt}}m{#1}}
\newcolumntype{R}[1]{>{\raggedleft\let\newline\\\arraybackslash\hspace{0pt}}m{#1}}

\makeatletter
\providecommand*{\boxast}{%
  \mathbin{
    \mathpalette\@boxit{*}%
  }%
}
\newcommand*{\@boxit}[2]{%
  \sbox0{$\m@th#1\Box$}%
  \ifx#1\displaystyle \ht0=\dimexpr\ht0+.05ex\relax \fi
  \ifx#1\textstyle \ht0=\dimexpr\ht0+.05ex\relax \fi
  \ifx#1\scriptstyle \ht0=\dimexpr\ht0+.05ex\relax \fi
  \ifx#1\scriptscriptstyle \ht0=\dimexpr\ht0+.05ex\relax \fi
  \sbox2{$#1\vcenter{}$}
  \rlap{%
    \hbox to \wd0{%
      \hfill
      \raisebox{%
        \dimexpr.25\dimexpr\ht0+\dp0\relax-\ht1\relax
      }{$\m@th#1#2$}%
      \hfill
    }%
  }%
  \Box
}
\makeatother

\renewcommand{\frak}{\mathfrak}

\DeclareMathOperator{\GL}{GL}
\DeclareMathOperator{\PGL}{PGL}
\DeclareMathOperator{\SL}{SL}
\DeclareMathOperator{\Ad}{Ad}
\DeclareMathOperator{\diag}{diag}

\DeclareMathOperator{\tr}{tr}

\renewcommand{\(}{\left(}
\renewcommand{\)}{\right)}
\def\rddots{\displaystyle\cdot^{\displaystyle\cdot^{\displaystyle\cdot}}}

\renewcommand\a{\alpha}
\renewcommand\b{\beta}
\newcommand\g{\gamma}
\renewcommand\d{\delta}
\newcommand\e{\varepsilon}
\renewcommand\l{\lambda}

\newcommand\D{\Delta}

\renewcommand\D{\Delta}
\newcommand\G{\Gamma}
\newcommand\f{\frac}
\newcommand\smallf[2]{{\textstyle{\frac{#1}{#2}}}}

\newcommand{\Z}{{\mathbb{Z}}}
\newcommand{\R}{{\mathbb{R}}}

\newcommand{\C}{{\mathbb{C}}}
\newcommand{\A}{{\mathbb{A}}}
\newcommand{\Q}{{\mathbb{Q}}}

\renewcommand\Re{\text{Re~}}

\renewcommand\i{^{-1}}
\renewcommand\({\left(}
\renewcommand\){\right)}
\newcommand{\ttwo}[4]{
\(\begin{smallmatrix}{#1} & {#2}
\\ {#3} & {#4} \end{smallmatrix}\)}

\newcommand{\sgn}{\operatorname{sgn}}

\newcommand{\gobble}[1]{}
  \newcommand{\rangeref}[2]{%
    \ref{#1}--\afterassignment\gobble\fam 0\ref{#2}%
  }

\makeatletter
\def\imod#1{\allowbreak\mkern5mu({\operator@font mod}\,#1)}
\makeatother

\begin{document}
\thispagestyle{empty}

$ $ \vspace {-2.2cm}
\begin{center}
\rule{8.8cm}{0.5pt}
\\[-5pt]
 {\footnotesize Riv.\, Mat.\, Univ.\, Parma,\, Vol. {\bf n} \,(20xx), \,000-000}
\\[-10pt]
\rule{8.8cm}{0.5pt}
\end{center}
\vspace {1.7cm}

\begin{center}
{\sc\large Dorian Goldfeld}, \
{\sc\large Stephen D.~Miller}  \ {\small and}  \
{\sc\large Michael Woodbury}
\end{center}
\vspace {1.1cm}

\centerline{\large{\textbf{A template method for Fourier coefficients  }}}
\centerline{\large{\textbf{  of Langlands Eisenstein series}}}

\renewcommand{\thefootnote}{\fnsymbol{footnote}}

\footnotetext{
Miller was supported by NSF Grant DMS-1801417. }
\footnotetext{
Goldfeld was  supported by Simons Collaboration Grant Number 567168. }

\renewcommand{\thefootnote}{\arabic{footnote}}
\setcounter{footnote}{0}

\vspace{0.6cm}
\begin{center}
\begin{minipage}[t]{11cm}
\small{
\noindent \textbf{Abstract.}
This paper introduces the template method for computing the first coefficient of Langlands Eisenstein series on $\GL(n,\mathbb R)$ and more generally on Chevalley groups over the adele ring of $\mathbb Q.$ In brief,  the first coefficient of Borel Eisenstein series   can be used as a template to compute the first coefficient of more general Eisenstein series by elementary linear algebra calculations.

\medskip

\medskip

\noindent \textbf{Mathematics~Subject~Classification~(2010):}
11S40, 11F70, 11F30.

}
\end{minipage}
\end{center}

\bigskip

\section{Introduction}\label{sec:intro}

According to the Merriam-Webster dictionary \cite{MerriamWebster2019}, a \emph{template} is ``a gauge, pattern, or mold used as a guide'' for producing other similar things.  In the theory of automorphic forms,   Eisenstein series, in particular those associated to the Borel subgroup, can sometimes serve as a template to deduce properties of other automorphic forms.

  A simple example of using Borel Eisenstein series (to be defined below) as a template is given as follows.  Consider the question of obtaining the precise form of the $\G$-factors in the functional equation of the standard  $L$-function of a globally spherical cusp form for $\GL(n)$ over $\Q$.  Once the $L$-function of spherical Borel Eisenstein series for $\GL(n)$ are shown to be   particular products of shifted Riemann $\zeta$-functions (and hence satisfy a specific functional equation with specific $\Gamma$-factors), it can then be  shown that cusp forms on $\GL(n)$ (with the same spectral parameters) will have the same  $\Gamma$-factors in their functional equation.  This method was employed in \cite{Bump84} for the case of $G=\GL(3)$, and it is described for $\GL(n)$ in \cite{Goldfeld2015} (see, in particular, Section~10.9).




The main aim of this paper is to show that  the Borel Eisenstein series can additionally be used as a template  to easily determine (in many cases) the non-constant Fourier coefficients of  Langlands Eisenstein series. As is well-known, the calculation of the non-constant  coefficients reaches a point where it depends only on local data.

However, until now this local calculation had not been directly performed  in the archimedean case, though it is possible to deduce the calculation using a result of Shahidi \cite{Shahidispaper} (which compares functionals on different automorphic representations) together with some facts from representation theory (specifically the points of reducibility of spherical principal series).  The importance of a direct proof is that anticipated analytic calculations in higher rank will require the same   specificity as has been crucially leveraged in low rank settings, such as fine control of the Whittaker functions.

 We give such a direct proof here in Theorem~\ref{thm:whitnorm}, by introducing  a notion of {\it canonically-normalized} Whittaker functions.  This latter aspect (which is our main new contribution) circumvents the problem that Jacquet's integral representation of Whittaker functions sometimes vanishes identically.  Our normalization does not, and has relatively simple asymptotics in the negative Weyl chamber.  Thus it provides an appropriate basis of special functions without which the Fourier series is somewhat awkward to even define.

Our explicit canonically-normalized Whittaker functions then  explicitly determine the generic Fourier coefficients, including the powers of $\pi$ and $\Gamma$-factors (which is our main contribution; the Euler factors have long been known -- see \cite[Theorem~7.1.2]{Shahidi2010}). Having such an exact formula is required in applications such as \cite{GSW2019}.
  Putting this together with the existing results at nonarchimedean places, we thus see that the local factors for the coefficients are mimicked by the local factors of the Borel series' coefficients.

To illustrate our method and results, we consider the special case of $G=\GL(3)$.  Let $\mathfrak{h}^3:=\GL(3,\R)/(O(3,\R)\cdot \R^\times)$, with the $\R^\times$ factor denoting the center of   $\GL(3,\R)$, i.e., scalar matrices, let $U_3$ be the subgroup of $\GL(3)$ consisting of upper triangular unipotent matrices, and let $T_3(\R)^0\subset \GL(3,\R)$ be the subgroup of diagonal matrices having positive entries on the diagonal.  Every coset representative in $\mathfrak{h}^3$ can be written as $g=xy$ with $x\in U_3(\R)$ and $y\in T_3(\R)^0/\R^\times_{>0}$.  We use the specific parametrization
\begin{equation}\label{eq:gequalsxy-GL3}
 x = \left(\begin{array}{ccc} 1 & x_{12} & x_{13} \\ & 1 & x_{23} \\ & & 1 \end{array}\right), \qquad y = \left( \begin{array}{ccc} y_1 y_2 \\ & y_1 \\ & & 1 \end{array}\right),
\end{equation}
(where $x_{12},x_{13},x_{23}\in\R, \, y_1,y_2>0$) to assign coordinates on $\mathfrak{h}^3$.

Let $\mathcal{D}^3$  denote the  invariant differential operators on $\mathfrak{h}^3$, namely all polynomials (with complex coefficients) in the variables $$\left\{ \frac{\partial}{\partial x_{12}},\quad \frac{\partial}{\partial x_{13}},\quad\frac{\partial}{\partial x_{23}}, \quad \frac{\partial}{\partial y_1},\quad\frac{\partial}{\partial y_2}\right\}$$ which are invariant under all $\GL(3,\R)$ transformations.
An important role is played by the eigenfunctions
\begin{equation}\label{Igalphadef}
I(g, \alpha) := I(xy, \alpha) = y_1^{1-\alpha_3}y_2^{1+\alpha_1},
\end{equation}
parametrized by triples $\alpha=(\alpha_1,\alpha_2,\alpha_3)\in\C^3$ satisfying $\alpha_1+\alpha_2+\alpha_3=0$.\footnote{
A different parametrization is used in some references, such as \cite{Bump84,Goldfeld2015}, where $\alpha$ is parameterized by  $\alpha_1 = 2v_1+v_2, \; \alpha_2 = -v_1+v_2, \; \alpha_3 = -v_1-2v_2$, with $v_1,v_2\in \C$.  Both the nature of our results as well as their general context makes it more convenient to work with a parametrization-free description of $\alpha$.}  The $I_\alpha$ are eigenfunctions of all operators in $\mathcal{D}^3$.  For example, the Laplace eigenvalue of $I(\cdot,\alpha)$ is  $1 - \frac{\alpha_1^2+\alpha_2^2+\alpha_3^2}{2}$  \cite[p.~49 and pp.~185-6]{Terras}.
It is a theorem of Harish-Chandra \cite{MR42422} that to any joint eigenfunction of the full ring of invariant differential operators $\mathcal{D}^3$, there exists such a triple $\alpha=(\alpha_1,\alpha_2,\alpha_3)$ sharing the same eigenvalues under each invariant differential operator in $\mathcal{D}^3$.  The triple $(\alpha_1,\alpha_2,\alpha_3)$ is unique up to permutation of the entries, and is known as the \emph{Langlands parameters} of the eigenfunction.

There are two types of Eisenstein series for $\GL(3)$, corresponding to whether one takes a minimal or maximal parabolic subgroup.  Let
 \[ \mathcal{B} = \left(\begin{array}{ccc} * & * & * \\ 0 & * & * \\ 0 &0  & * \end{array}\right) \qquad\mbox{and}\qquad \mathcal{P}_{2,1}=\left(\begin{array}{ccc} * & * & * \\ * & * & * \\ 0 & 0 & * \end{array}\right)  \]
 denote the Borel and a specific maximal parabolic, respectively.  Any parabolic subgroup of $\GL(3)$ is conjugate to a standard parabolic, (i.e., one containing $\mathcal{B}$), and the only other standard parabolics of $G$ are   $\mathcal P_{1,2}$ (defined analogously to $\mathcal{P}_{2,1}$, but instead with zero entries in the bottom two entries of the first column) and the full group $G$ itself.
The  Borel Eisenstein series for $G=\GL(3,\R)$  is defined for $\Re(\alpha_1-\alpha_2)$ and $\Re(\alpha_2-\alpha_3)>1$ by the absolutely convergent sum
\begin{equation}\label{gl3eisdefinintro}
  E_{\mathcal{B}}(g,\alpha) := \sum_{\gamma\in U_3(\Z)\backslash \SL(3,\Z)} I(\gamma g, \alpha), \qquad (g\in \frak h^3),
\end{equation}
and for general $\a$ by meromorphic continuation.

Eisenstein series for $\mathcal P_{2,1}$ are defined through automorphic forms on $\GL(2)$.
Let $\phi$ be a Maass cusp form for $\SL(2, \Z)$ with spectral parameter $v\in\C$ (i.e., the Laplace eigenvalue of $\phi$ is $\f 14-v^2$ under the usual normalization).  Another type of Langlands Eisenstein series, denoted  $E_{\mathcal{P}_{2,1},\phi}(\cdot, s)$, can be created  for the parabolic subgroup $\mathcal P_{2,1}$.  It is twisted by the Maass form $\phi$ and  defined by summing a function formed from $\phi$, over $\mathcal{P}_{2,1}(\Z)\backslash \SL(3,\Z)$ -- see Definition~\ref{EisensteinSeries}.

Each $M=(m_1,m_2)\in \Z^2$ with $m_1 m_2\neq 0$ determines a nondegenerate character
 \[ \psi_M: U_3(\R)\to \C^\times, \qquad \psi_M\left(x\right) := e^{2\pi i(m_1x_{12}+m_2x_{23})}, \]
where $x_{12}$ and $x_{23}$ are the super  diagonal entries of $x\in U_3(\R)$ as in \eqref{eq:gequalsxy-GL3}.  Using this, one obtains Fourier coefficients
 \[ \int\limits_{U_3(\Z)\backslash U_3(\R)} f(ug) \overline{ \psi_M(u)}\; du \]
for any $\GL(3,\Z)$ automorphic form $f$.

In particular, if $f=E$ is any one of the Eisenstein series defined above, then for $s\in\C$ and $M=(m_1,m_2)$ as above, we have
 \[ \int\limits_{U_3(\Z)\backslash U_3(\R)} E(ug,s) \overline{\psi_M(u)}\; du = \frac{A_E(M,s)}{|m_1m_2|} W_\alpha(Mg), \]
where $W_\alpha$ is a Whittaker function (specifically, the canonically-normalized Whittaker function from Section~\ref{sec:whit}) and $A_E(M,s)$ is termed the $M^{th}$ arithmetic Fourier coefficient of $E$.  The {\it first coefficient}  of $E$ is defined to be  $A_E((1,1),s)$.
Furthermore,  we have
 \[ A_E(M,s) = A_E((1,1),s)\cdot \lambda_E(M,s). \]
where $\lambda_E(M,s)$ is the $M^{th}$ Hecke eigenvalue of $E$ and $\lambda_E((1,1),s)=1$,

The first coefficient  of the Borel Eisenstein series $E_{\mathcal{B}}(g,\alpha)$ only involves the completed Riemann zeta function $\zeta^*(w)$ (for some $w\in\C$ depending on $s$) where
\begin{equation}\label{zetadef}
  \zeta^*(w) := \pi^{-\frac{w}{2}} \Gamma\left(w/2 \right)\zeta(w) = \zeta^*(1-w).
\end{equation}
Let $\phi$ be a Maass form for $\SL(2, \mathbb Z)$ with Laplace eigenvalue $\frac14-v^2.$
Then the first coefficient of $E_{\mathcal{P}_{2,1},\phi}(g,s)$ involves both the completed $L$-function attached to the Maass form  $\phi$ given by
\begin{align*}
L^*(w,\phi) & := \pi^{-w} \Gamma\left( \frac{w+v}{2}  \right)\Gamma\left( \frac{w-v}{2}  \right) L(w, \phi)\\
& = L^*(1-w,\phi),
\end{align*}
and the special value of the completed adjoint $L$-function $L^*(1, \text{Ad}\; \phi)$ (see Proposition \ref{FirstCoeffMaassForm}).

The determination of the constant Fourier coefficients of Eisenstein series was given in great generality for arbitrary reductive groups in Langlands \cite{Langlands1976}.\footnote{In fact,   constant terms can be treated using a template method of their own (following the same strategy we use here, but greatly simplified because there are no complicated special functions involved, unlike for non-constant Fourier coefficients as we face in Section~\ref{sec:whit}).  Instead of (\ref{constterm1}) one uses the calculation found on  \cite[p.~92]{MW1995},  at which point a formula analogous to (\ref{constterm2}) boils down to the local calculations already found in \cite{Langlands1971} for maximal parabolic Eisenstein series.}  The precise determination of all Fourier coefficients of the $\GL(3)$  Borel Eisenstein series $E_{\mathcal{B}}(g,s)$ was obtained independently by Bump \cite{Bump84}, Vinogradov and Takhtadshyan \cite{Vin1978}, and Imai and Terras \cite{IT1982}.
The precise determination of the first coefficients $A_{E_{\mathcal B}}((1,1),s)$ and $A_{E_{\mathcal P_{2,1},\phi}}((1,1),s)$ (up to a non-zero constant factor) is summarized in the following theorem.

\begin{theorem}\label{thm:inintro}  1) Let  $\alpha = (\alpha_1,\alpha_2,\alpha_3)\in \mathbb C^3$ with $\alpha_1+\alpha_2+\alpha_3=0.$
The first coefficient $A_{E_{\mathcal B}}((1,1),\alpha)$  of the $\GL(3)$ Borel Eisenstein series is equal to
\begin{equation}\label{eq:BorelFirstCoeff} \Big( \zeta^*\big(1+\alpha_1-\alpha_2\big) \zeta^*\big(1+\alpha_2-\alpha_3\big) \zeta^*\big(1+\alpha_1-\alpha_3\big) \Big)^{-1}.
\end{equation}

2)
Let $s = \left( s_1,\, -2s_1\right)$ and let $\phi$ be a Maass form for $\SL(2,\Z)$ with Petersson norm one. Then the first coefficient
 $A_{E_{\mathcal P_{2,1},\phi}}((1,1),s)$ of the $\GL(3)$ Eisenstein  series $E_{\mathcal{P}_{2,1},\phi}(g,s)$ induced from $\phi$ (see Definition~\ref{EisensteinSeries}) is equal to
\begin{equation}\label{eq:GL3FirstCoeff}\Big(L^*(1,\Ad{\phi})^{1/2} \cdot  L^*(\phi,1 + 3s_1)\Big)^{-1}
\end{equation}
times an explicit, absolute   non-zero constant.
\end{theorem}
\noindent  A sketch of the proof is given in \S \ref{sec:sketch}  and a more complete proof is given in the greater generality of Chevalley groups in Theorem~\ref{thm:templategeneral} and the computations performed in the first two examples of Section~\ref{sec:examplesoftemplatemethod}.  With this example in mind, we can now explain the template method, in which the main idea is that Fourier coefficients such as (\ref{eq:GL3FirstCoeff}) can be derived from those from minimal parabolics, such as  (\ref{eq:BorelFirstCoeff}).

\subsection*{The Template Method:}\label{TheTemplateMethod}
{\leftskip= 1cm\relax
 \rightskip=1cm\relax
 \noindent
{\bf Step 1:} {\bf Replace a cusp form with a (smaller) Borel Eisenstein series. } Here we form a new Eisenstein series $E_{\mathcal B, \text{new}}$ by following the recipe in Definition~\ref{EisensteinSeries}, but instead of a cusp form $\Phi$ on the Levi of $\mathcal P$ we use a Borel Eisenstein series for that Levi. We choose the parameters of the Levi Eisenstein series to match the Langlands parameters of $\Phi$. The resulting maximal parabolic Eisenstein series on $\GL(n)$ induced from the Levi Eisenstein series is, in fact, equal to a Borel Eisenstein series, with   Langlands parameters given in Definition \ref{MinParEisSeries}.  A similar matching is possible at $p$-adic places using Satake parameters (see Definition~\ref{def:satake}).
\par}

\vskip 8pt
{\leftskip= 1cm\relax
 \rightskip=1cm\relax
 \noindent
 {\bf Step 2:} {\bf We use an analog of (\ref{eq:BorelFirstCoeff}) to compute the first Fourier coefficient of the Borel Eisenstein series $E_{\mathcal B, \text{new}}$, }which involves the product $\prod\limits_{1\le i<j\le n}\zeta^*(1+\a_i-\a_j)\i$ in terms of the Langlands/Satake parameters $(\a_1,\ldots,\a_n)$.
 \par}
 \vskip 8pt
 {\leftskip=1cm\relax
 \rightskip=1cm\relax
 \noindent
 {\bf Step 3:} The Fourier coefficient of the cuspidally-induced Eisenstein series now involves Euler factors matching those of the above expression (in which the Langlands/Satake parameters take different values for different primes in the Euler product), but omitting any pair of indices $(i,j)$  for which the Levi of $\mathcal P$ has a $\GL(n_i)$-factor containing the $(i,j)$-th position.
\par}
\vskip10pt
Of course this sketch does not fully address normalizations, which are handled in Theorem~\ref{thm:templategeneral}.  There are typically two standard ways to normalize automorphic forms, when possible:~either by  their first Fourier coefficient or by their $L^2$ norm.  The exact relationship involves
    the normalizing factor involving $L(1,\Ad{\phi})^{1/2}$ in (\ref{eq:GL3FirstCoeff}), and is crucial for applications (e.g., coming from the trace formula).

\vskip 10pt
 The main goal of this paper is to generalize Theorem~\ref{thm:inintro} and extend the Template Method to arbitrary Chevalley groups, which is done in Theorem~\ref{thm:templategeneral}. With this later result in hand, the Template Method becomes an algorithm to explicitly determine the first coefficient of Langlands Eisenstein series.  Matching  the Langlands parameters in step 1)  involves only simple linear algebra (which can also be expressed in the language of root systems).  In Part I of this paper, the Template Method is first developed in the case of $\GL(n)$ \textcolor{black}{with explicit examples worked out in the table of Fourier coefficients of $\SL(4,\mathbb Z)$ in \S 5 following the classical language in \cite{Goldfeld2015}.} We then show Part II of the paper that the principle applies to and can be proved in the general setting \textcolor{black}{of Chevalley groups.}
\vskip 6pt

We now briefly summarize Part II of this paper. In \S 6 we present an elementary introduction to roots and powers via the example of the Chevalley group $\SL(3).$ This is followed by a general review of Chevalley groups over local fields and adeles. The Langlands Eisenstein series for Chevalley groups are then defined. In \S 7 Whittaker functions on Chevalley groups are introduced and canonical normalizations are obtained via Jacquet integrals in \S 8, with Theorem~\ref{thm:whitnorm} being the main novel mathematical contribution of the paper.

In Theorem \ref{templateglobal} explicit formulae for the Fourier coefficients of cuspidally-induced Eisenstein series on Chevalley groups (in terms of local data consisting of Satake/Langlands parameters)  are obtained, providing justification for the template method. Such formulae had been known at the finite places, but the main contribution of this paper is to obtain the archimedean factors.  That involves the explicit canonical normalization of the spherical archimedean Whittaker functions provided in Theorem~\ref{thm:whitnorm}.  Our definition is the multiple
\begin{equation}\label{Winfcanonintro}
  W_{\infty,\lambda}^{canon}(g)  = \(\prod_{\alpha\in\Delta_+}
 \Gamma_\R\Big( \langle \lambda,\alpha^\vee \rangle+1\Big)\) \operatorname{Jac}_\lambda(g)
\end{equation}
of Jacquet's integral representation of the Whittaker function.  While the latter can vanish identically in $g$ for certain parameters $\l$, the canonical normalization never does; furthermore, it possesses a natural functional equation and straightforward asymptotics (see Section~\ref{sec:whit} for notation and more details).
 Finally, in \S 9 several examples of the template method in the general setting of Chevalley groups are given.

\vskip 10pt

The calculation of the non-constant Fourier coefficients of Eisenstein series is important in many applications, ranging from number theory to string theory.  In many previous applications it has been sufficient to compute coefficients up to undetermined, nonzero scalars.  However, in other applications (e.g., string instantons  and to spectral theory calculations in automorphic forms) the actual value of the constant, which we provide here, is very important (see \cite{blomer}, \cite{GK12}, \cite{FGKP} and \cite{GSW2019}).

\part{Template method for $\GL(n)$}

This first part of the paper states the main results for $\GL(n)$, along with a description of the techniques, main objects, and arguments in classical terminology where appropriate.  Generalizations to general Chevalley groups (and using adeles) are given in Part II.

\section{\bf Basic notation and definitions}
\label{sec:Basicnotation}

  Let $n\ge 2$ be an integer and define
  $\mathfrak h^n := \GL(n,\mathbb R)/\left(O(n,\mathbb R)\cdot\mathbb R^\times   \right)$, with $\R^\times$ denoting $\GL(n,\R)$'s center of scalar matrices.
  Every element of $\mathfrak h^n$ has an upper-triangular coset representative of the form $g=x y$, with
$$x =  \left(\begin{smallmatrix} 1 & x_{1,2}& x_{1,3}& \cdots  & &
x_{1,n}\\
 & 1& x_{2,3} &\cdots & & x_{2,n}\\
& &\hskip 2pt \ddots & & & \vdots\\
& && & 1& x_{n-1,n}\\
& & & & &1\end{smallmatrix}\right), \;\;\;\;\;\;\;\;
y =
\left(\begin{smallmatrix} y_1y_2\cdots
y_{n-1} & & & \\
& \hskip -30pt y_1y_2\cdots y_{n-2} & & \\
& \ddots &  & \\
& & \hskip -5pt y_1 &\\
& & &  1\end{smallmatrix}\right),$$
 $x_{i,j} \in \mathbb R$ ($1\le i < j\le n$), and $y_i > 0$ ($1 \le i \le n-1$).  The group $\GL(n,\R)$ acts as a group of transformations on $\mathfrak h^n$ by left-multiplication.

\begin{definition}{\bf (The space of invariant differential operators)} Let $n \ge 2.$ The ring of invariant differential operators, denoted
 $\mathcal D^n$, consists of all polynomials (with complex coefficients)  in the differential operators
 $\frac{\partial}{\partial x_{i,j}}$  $(1\le i < j\le n)$ and
 $\frac{\partial}{\partial y_i}$ $(1 \le i \le n-1)$ which are invariant under all $\GL(n,\mathbb R)$ transformations.
\end{definition}

\begin{definition} {\bf (Langlands parameters for $\GL(n)$)}\label{def:langlandsparameters}
Let $n\geq 2$.  Langlands parameters are complex numbers $\{\alpha_1,\ldots,\alpha_n\}$ which satisfy $\alpha_1+\cdots+\alpha_n=0$.  By abuse of notation we often refer to the Langlands parameters as a vector $\alpha=(\alpha_1,\ldots,\alpha_n)$ (up to arbitrary permutations of the coordinates).
\end{definition}

\begin{remark}
These parameters are used to classify automorphic representations at the archimedean place.
\end{remark}

\begin{definition} {\bf (The eigenfunction $I(\cdot,\alpha)$)}
Let $n\geq 2$.  Let $\alpha = (\alpha_1,\ldots,\alpha_n)$ denote Langlands parameters and let $\rho=(\rho_1,\ldots,\rho_n)$, where $\rho_i=\frac{n+1}{2}-i$.   We define a power function  on $xy\in \mathfrak h^n$  by
\begin{equation}\label{Ialpha}
I(xy,\alpha) =\prod_{i=1}^n d_i^{\alpha_i+\rho_i}= \prod_{i=1}^{n-1}y_i^{\alpha_1+\cdots+\alpha_{n-i}+\rho_1+\cdots+\rho_{n-i}},
\end{equation}
where $d_i=\prod\limits_{1\le j\le n-i}y_j$ is the $j$-th diagonal entry of the   upper-triangular matrix $g=xy$ as above. The function $I(\cdot,\alpha)$ is an eigenfunction of $\mathcal{D}^n$.
\end{definition}

\begin{definition} {\bf (Langlands parameters of joint eigenfunctions of $\mathcal{D}^n$)}
It is a theorem of Harish-Chandra \cite{MR42422} that to any joint eigenfunction of the full ring of invariant differential operators $\mathcal{D}^n$, there exists Langlands parameters $\alpha=(\alpha_1,\ldots,\alpha_n)$ sharing the same eigenvalues $\l_\d(\a)$ under each invariant differential operator $\d\in \mathcal{D}^n$.  The $n$-tuple $(\alpha_1,\cdots,\alpha_n)$ is unique up to permutation of the entries, and is known as the Langlands parameters of the eigenfunction.
\end{definition}

The above notation in terms of the Langlands parameters has the merit of generalizing nicely to Chevalley groups.  However, several authors have used a different convention with analytic applications in mind, in which it is convenient to write $I(g,\alpha)$ as the following power function $I_s(g)$:

  \begin{definition}\label{def:conversion} {\bf (The eigenfunction $I_s$)}
   Let $n\ge 2$. Define $$s = (s_1, s_2, \ldots,s_{n-1}) \in \mathbb C^{n-1},$$ and also define the power function
${I_s}:U_n(\mathbb R)\backslash \mathfrak h^n\to\mathbb C$ by
\begin{align*}
&\hskip 40pt  I_s(g)  := \prod_{i=1}^{n-1} \prod_{j=1}^{n-1}
y_i^{b_{i,j}s_j}, \qquad\quad \left(g = xy\in\mathfrak h^n\right),
\end{align*}
where
$$b_{i,j} = \begin{cases} \; i \cdot j & \text{if $i + j \le n,$}\\
(n-i)(n-j) & \text{if $i + j \ge n.$}\end{cases}
$$
 \end{definition}

This parametrization is chosen to simplify later formulae.  The genesis of the formula for $b_{i,j}$ is
the Lie-theoretic fact that the inverse of the Cartan matrix of $\GL(n)$ is the $(n-1)\times (n-1)$ matrix whose $(i,j)$-th entry is $\frac{b_{n-i,j}}{n}$.  It is simple to see that with $$s_i=\frac{\alpha_i-\alpha_{i+1}+1}{n},$$ one has $I(g,\alpha)=I_s(g)$.  In the reverse direction, $\alpha=\alpha(s)$ can be recovered by the formula
\begin{equation}\label{eq:alphasformula}
 \alpha_i(s) := \begin{cases} B_{n-1}(s) & \text{if}\; i=1,\\
B_{n-i}(s) - B_{n-i+1}(s) & \text{if}\; 1<i<n,\\
-B_1(s) & \text{if} \; i = n,    \end{cases}
\end{equation}
where $B_j(s) =\sum\limits_{i=1}^{n-1} b_{i,j} (s_i-\frac{1}{n}).$

 For example, in the special case of $\SL(3, \mathbb Z)$  the Langlands parameters $\alpha(s)$ associated to  $s = \left(v_1+\frac{1}{3},v_2+\frac{1}{3}\right)$ are given by
\begin{equation}\label{LanglandsParameters-SL(3,Z)}
\alpha_1(s) = 2v_1+v_2, \quad \alpha_2(s)= -v_1+v_2, \quad \alpha_3(s) = -v_1-2v_2.
\end{equation}
while, in the special case of $\SL(4, \mathbb Z)$, the Langlands parameters $\alpha(s)$ associated to  $s = \left(v_1+\frac{1}{4},v_2+\frac{1}{4}, v_3+\frac{1}{4}\right)$ are given by
\begin{align}\label{LanglandsParameters-SL(4,Z)}
&
\alpha_1(s) = 3v_1+2v_2+v_3, \quad \, \alpha_2(s) = -v_1+2v_2+v_3,
\\
& \alpha_3(s) = -v_1-2v_2+v_3,\quad \alpha_4(s) = -v_1-2v_2-3v_3.
\nonumber\end{align}
Similarly, for $n\ge 2$ one can define spectral parameters  $v_i=s_i-\frac{1}{n}$, for $1\le i \le n-1$.

\begin{definition} {\bf  (Maass cusp forms for $\SL(n,\mathbb Z)$)}\label{def:MaassForm}
Let $n\ge 2.$ A Maass cusp form $\phi: \SL(n,\Z)\backslash \mathfrak{h}^n \to \mathbb C$ for $\SL(n,\mathbb Z)$ is a joint eigenfunction of  $\mathcal D^n$ which satisfies a bound of the form
$$|\phi(xy)|\ll_N (y_1\cdots y_{n-1})^{-N}$$ in the range $y_1,\ldots,y_{n-1}\ge 1$, for any $N>0$.  If the eigenvalues of $\phi$ agree with those of $I(\cdot,\alpha)$ then we term $\alpha$ the Langlands parameters of $\phi$. The Maass form is called a Hecke eigenform if it is a simultaneous eigenfunction of all the Hecke operators.
\end{definition}

 If $\lambda_\Delta$ denotes the Laplace eigenvalue of a Maass cusp form $\phi$  then (like all eigenfunctions),  $\lambda_\Delta$ can be expressed in terms of its Langlands parameters $\alpha = \left(\alpha_1, \ldots, \alpha_n\right) \in \mathbb C^n$:
     $$\lambda_\Delta =  \frac{n^3 - n}{24} -\frac{\alpha_1^2+\alpha_2^2+\cdots+\alpha_n^2}{2}$$
     (see \cite[p.~49 and pp.~185-6]{Terras}).
The generalized Ramanujan-Selberg conjectures assert that all Maass cusp forms for $\SL(n,\Z)$ (and more generally for any congruence subgroup of $\SL(n,\Z)$ are tempered, that is, their Langlands parameters are all purely imaginary.

\begin{definition} {\bf (Parabolic Subgroup)} \label{GLnParabolic} For $n\ge 2$ and $1\le r\le n,$ consider a partition of $n$ given by
$n = n_1+\cdots +n_r$ with positive integers $n_1,\cdots,n_r.$  We define the standard parabolic subgroup $$\mathcal P := \mathcal P_{n_1,n_2, \ldots,n_r} := \left\{\left(\begin{matrix} \GL(n_1) & * & \cdots &*\\
0 & \GL(n_2) & \cdots & *\\
\vdots & \vdots & \ddots & \vdots \\
0 & 0 &\cdots & \GL(n_r)\end{matrix}\right)\right\},$$
where the notation indicates that the $n_i\times n_i$ diagonal entries can be arbitrary elements of $\GL(n_i)$.

Letting $I_r$ denote the $r\times r$ identity matrix, the subgroup
$$N^{\mathcal P} := \left\{\left(\begin{matrix} I_{n_1} & * & \cdots &*\\
0 & I_{n_2} & \cdots & *\\
\vdots & \vdots & \ddots & \vdots \\
0 & 0 &\cdots & I_{n_r}\end{matrix}\right)\right\}$$
is the unipotent radical of $\mathcal P$.  The subgroup
$$M^{\mathcal P} := \left\{\left(\begin{matrix} \GL(n_1) & 0 & \cdots &0\\
0 & \GL(n_2) & \cdots & 0\\
\vdots & \vdots & \ddots & \vdots \\
0 & 0 &\cdots & \GL(n_r)\end{matrix}\right)\right\}$$
(with the same convention that the diagonal entries are invertible $n_i\times n_i$ matrices) is the standard choice of Levi subgroup of $\mathcal P$.
\end{definition}

For example, when $n$ is partitioned into $r=n$ pieces, each of which has $n_i=1$, then  $\mathcal P$ specializes to the Borel subgroup
$$\mathcal B=\mathcal P_{\text Min}=\left\{\left(\begin{smallmatrix}
*&* &\cdots &*\\
& * &\cdots & *\\
&&\ddots &\vdots \\
& & & *\end{smallmatrix}\right) \subset \GL(n,\mathbb R)   \right\},$$ i.e., all upper triangular matrices,  which is the minimal standard parabolic subgroup.  It has $N^{\mathcal B}=U_n$, the group of all upper triangular matrices having ones on the diagonal, and $M^{\mathcal B}=T$, the torus of invertible  diagonal matrices.

By the Iwasawa decomposition, every $g \in \GL(n,\mathbb R)$ can be expressed in the form
$g = u t k$ where $u \in U_n(\mathbb R)$, $t \in T(\mathbb R)$, and $k\in K = O(n,\mathbb R)$, the maximal compact subgroup of $\GL(n,\mathbb R)$.  Then the element $t$ lies in every standard parabolic $\mathcal P$ and may be expressed in the form
\begin{equation} \label{ToricParabolicElement}
t = \left(\begin{matrix} t_1 & 0 & \cdots &0\\
0 & t_2 & \cdots & 0\\
\vdots & \vdots & \ddots & \vdots \\
0 & 0 &\cdots & t_r\end{matrix}\right),
\end{equation}
where each $t_i$ is a diagonal matrix in $\GL(n_i).$  More generally, since $\GL(n,\mathbb R)=\mathcal B(\mathbb R)K$ and since any parabolic $\mathcal P\supset \mathcal B$, one also has $\GL(n,\mathbb R)=\mathcal P(\mathbb R)K$ (without uniqueness of expression).

\vskip 8pt

We conclude this section with a sequence of definitions, culminating in the construction of Eisenstein series.

 \begin{definition}\label{def:maassparabolic} {\bf (Maass form $\Phi$ associated to a parabolic $\mathcal P$)} \label{InducedCuspForm} Let $n\ge 2$. Consider a partition $n = n_1+\cdots +n_r$ with $1 < r < n$. Let  $\mathcal P := \mathcal P_{n_1,n_2, \ldots,n_r} \subset \GL(n,\mathbb R).$
  For $i = 1,2,\ldots, r$, let
$\phi_i:\GL(n_i,\mathbb R)\to\mathbb C$ be either the constant function 1 (if $n_i=1$) or a Maass cusp form for $\SL(n_i,\mathbb Z)$ (if $n_i>1$).  The form $\Phi := \phi_1\otimes \cdots\otimes \phi_r$ is defined on $\GL(n,\mathbb R)=\mathcal P(\mathbb R)K$ by the formula
$$\Phi(n m k ) := \prod_{i=1}^r \phi_i(m_i), \qquad (n\in  N^{\mathcal P}, m\in  M^{\mathcal P},k\in K)$$
where
  $m \in M^{\mathcal P}$ has the form
$m = \left(\begin{smallmatrix} m_1 & 0 & \cdots &0\\
0 & m_2 & \cdots & 0\\
\vdots & \vdots & \ddots & \vdots \\
0 & 0 &\cdots & m_r\end{smallmatrix}\right)$, with    $m_i\in \GL({n_i},\mathbb R).$  In fact, this construction works equally well if some or all of the $\phi_i$ are Eisenstein series.

 \end{definition}

%
%
%

\begin{definition} \label{MinParEisSeries} {\bf (Minimal parabolic Eisenstein series)}  Let $n \ge 2.$  The minimal parabolic (or Borel) Eisenstein series for $\SL(n,\mathbb Z)$ is defined by the formula
$$E_{\mathcal P_{\text{\rm Min}}}(g, \alpha) := \sum_{\gamma\, \in \,\left(\mathcal P_{\text{\rm Min}}\cap \Gamma    \right)\backslash\Gamma} I(\gamma g,\alpha), \qquad (g\in \GL(n,\mathbb R)),$$
initially as an absolutely convergent sum for those Langlands parameters $\alpha=(\alpha_1,\ldots,\alpha_n)$ satisfying $\Re(\alpha_i-\alpha_{i+1})>1$, and then by meromorphic continuation to all $\alpha\in \C^n$ satisfying $\alpha_1+\cdots+\alpha_n=0$.  Note that in terms of the parametrization of $\alpha$ by the variable $s=(s_1,\ldots,s_n)$ given in \eqref{eq:alphasformula},
 \[ E_{\mathcal P_{\text{\rm Min}}}(g, \alpha(s)) = \sum_{\gamma\, \in \,\left(\mathcal P_{\text{\rm Min}}\cap \Gamma    \right)\backslash\Gamma} I_s(\gamma g). \]
\end{definition}

\begin{remark}
For fixed Langlands parameters $\alpha$, the  Eisenstein series $E_{\mathcal{P}_{\mathrm{Min}}}$ is an eigenfunction for all $\mathcal{D}^n$ with the same eigenvalues as $I(\cdot,\alpha)$.  Therefore, $E_{\mathcal{P}_{\mathrm{Min}}}$ has Langlands parameters $\alpha$.
\end{remark}


%
%

 Next, we define the Eisenstein series for $\Gamma$ that are twisted by Maass forms of lower rank, which involve multiplying $\Phi$ from Definition~\ref{def:maassparabolic} by a character:

\begin{definition} {\bf (Character of a parabolic subgroup)} \label{ParabolicNorm}  Let $n\ge 2.$ Fix a partition
$n = n_1+n_2 + \cdots +n_r$ with associated parabolic subgroup $\mathcal P := \mathcal P_{n_1,n_2, \ldots,n_r}.$  Define
\begin{equation}\label{rhoPj}
\rho_{_{\mathcal P}}(j)=\left\{
                            \begin{array}{ll}
                              \f{n-n_1}{2}, & j=1 \\
                              \f{n-n_j}{2}-n_1-\cdots-n_{j-1}, & j\ge 2.
                            \end{array}
                          \right.
\end{equation}
 Let $s = (s_1,s_2, \ldots, s_r) \in\mathbb C^r$ satisfy
$\sum\limits_{i=1}^r n_i s_i = 0.$
Consider the function
$$\boxed{\; | \cdot   |_{_{\mathcal P}}^s := I(\cdot,\alpha) }$$
 on $\GL(n,\mathbb R)$, where
\begin{multline*}
\alpha=(\overbrace{s_1-\rho_{_{\mathcal P}}(1)+\smallf{1-n_1}{2}, \; s_1-\rho_{_{\mathcal P}}(1)+\smallf{3-n_1}{2}, \; \ldots \; ,s_1-\rho_{_{\mathcal P}}(1)+\smallf{n_1-1}{2}}^{n_1 \;\,\text{\rm terms}},\\
\overbrace{s_2-\rho_{_{\mathcal P}}(2)+\smallf{1-n_2}{2}, \; s_2-\rho_{_{\mathcal P}}(2)+\smallf{3-n_2}{2}, \; \ldots \; ,s_2-\rho_{_{\mathcal P}}(2)+\smallf{n_2-1}{2}}^{n_2 \;\,\text{\rm terms}},\\
\vdots
\\
\ldots \;\; ,\overbrace{s_r-\rho_{_{\mathcal P}}(r)+\smallf{1-n_r}{2}, \; s_r-\rho_{_{\mathcal P}}(r)+\smallf{3-n_r}{2}, \; \ldots \; ,s_r-\rho_{_{\mathcal P}}(r)+\smallf{n_r-1}{2}}^{n_r \;\,\text{\rm terms}}).
\end{multline*}
The conditions $\sum\limits_{i=1}^r n_i s_i = 0$ and $\sum\limits_{i=1}^r n_i \rho_{_{\mathcal P}}(i) = 0$
 guarantee that  $\alpha$'s entries sum to zero.  When $g\in \mathcal P$, with diagonal block entries $m_i\in \GL(n_i,\mathbb R)$, one has
 \vskip-10pt
 $$| g  |_{_{\mathcal P}}^s=\prod_{i=1}^r\left| \text{\rm det}(m_i)\right|^{s_i},$$
  \vskip-10pt \noindent
   so that $| \cdot   |_{_{\mathcal P}}^s$ restricts to a character of $\mathcal P$ which is trivial on $N^{\mathcal P}$.
 \end{definition}

For example, when $n=3$, $n_1=2$, and $n_2=1$, $\rho_{_{\mathcal P}}(1)=\f 12$ and $\rho_{_{\mathcal P}}(2)=-1$.  The parametrization of the $s_i$ here is chosen to match that of Definitions 10.4.5 and 10.5.3 in \cite{Goldfeld2015}.

 \begin{definition} {\bf (Langlands Eisenstein series twisted by Maass forms of lower rank) } \label{EisensteinSeries}
Let $\Gamma = \SL(n, \mathbb Z)$ with $n \ge 2.$  Consider a parabolic subgroup $\mathcal P$ of $\GL(n,\mathbb R)$ and functions $\Phi$  and $| \cdot |_{_{\mathcal P}}^s$   as given in Definitions \ref{InducedCuspForm} and  \ref{ParabolicNorm}, respectively.
The Langlands Eisenstein series determined by this data is defined by
\begin{equation}\label{def:EPhi}
 E_{\mathcal P, \Phi}(g,s) := \sum_{\gamma\,\in\, (\mathcal P \cap \Gamma)\backslash \Gamma}  \Phi(\gamma g)\cdot |\gamma g|^s_{_{\mathcal P}}
 \end{equation}
as an absolutely convergent sum for $\Re(s_i)$ sufficiently large, and extends to  $s\in \C^r$ by meromorphic continuation.
\end{definition}

If the Langlands parameters of $\phi_i$ are written as $(\alpha_{i,1},\ldots,\alpha_{i,n_i})$,
the Langlands parameters of $E_{\mathcal P, \Phi}(g,s)$ are

\begin{multline}\label{langlandsparamsforPhi}
 (\overbrace{\alpha_{1,1}+s_1-\rho_{_{\mathcal P}}(1), \qquad\ldots \qquad,\alpha_{1,n_1}+s_1-\rho_{_{\mathcal P}}(1)}^{n_1 \;\,\text{\rm terms}},\\
 \hskip40pt
\overbrace{\alpha_{2,1}+s_2-\rho_{_{\mathcal P}}(2), \qquad\ldots \qquad,\alpha_{2,n_2}+s_2-\rho_{_{\mathcal P}}(2)}^{n_2 \;\,\text{\rm terms}},\\
\vdots\\
\ldots, \quad
\overbrace{\alpha_{r,1}+s_r-\rho_{_{\mathcal P}}(r), \qquad\ldots \qquad,\alpha_{r,n_r}+s_r-\rho_{_{\mathcal P}}(r)}^{n_r \;\,\text{\rm terms}}).
\end{multline}

\section{\bf Fourier expansions of Maass forms for $\SL(n, \mathbb Z)$}

\begin{definition} {\bf (Character of the unipotent group $U_n$)}
   Let $n \ge 2.$ For  $L = (\ell_1, \ell_2, \ldots, \ell_{n-1}) \in \Bbb Z^{n-1},$ the function $\psi_L:U_n(\mathbb R)\rightarrow\C$ defined by
   $$\boxed{ \psi_{L}(u) := e^{2\pi i\big(\ell_1u_{1,2} + \; \cdots \; + \ell_{n-1}u_{n-1,n}\big)}, } \qquad u =  \left(\begin{smallmatrix} 1 & u_{1,2} & u_{1,3} & \cdots  & u_{1,n}\\
   & 1 & u_{2,3} & \cdots  & u_{2,n}\\
   & & \ddots& \ddots  &\vdots\\
   &&&1&u_{n-1,n}\\
   & & & & 1\end{smallmatrix}\right), $$
   is a multiplicative character.  It is {\em nondegenerate} when none of the integers $\ell_1,\ldots,\ell_{n-1}$ are zero.
   \end{definition}

In terms of the above coordinates, a bi-invariant Haar measure $du$ on $U_n(\mathbb R)$ is defined by $du := \prod\limits_{1\le i<j\le n}du_{i,j}$.

\begin{definition}\label{def:jacquetwhit} {\bf (Canonically Normalized  Whittaker Function)}
$\phantom{x}$
\vskip 3pt
Assume $n\ge 2$ and let $w_n = \left(\begin{smallmatrix} &&1\\
&\rddots&\\ 1&&
  \end{smallmatrix}   \right)$.
The function
   $$W^{\pm}_{\alpha}(g) := \prod_{1\leq j< k \leq n} \frac{\Gamma\left(\frac{1 + \alpha_j - \alpha_k}{2}\right)}{\pi^{\frac{1+\alpha_j - \alpha_k}{2}} }\cdot \int\limits_{U_n(\mathbb R)} I(w_n ug, \alpha) \,\overline{\psi_{1, \ldots, 1,\pm 1}(u)} \, du, $$
defined as an absolutely-convergent integral for $\Re(\alpha_i-\alpha_{i+1})>0$ (with $1\le i<n),$  extends to a holomorphic function on $\{\alpha \in \mathbb C^n|\alpha_1+\cdots+ \alpha_n=0\}$.
\end{definition}
This is an example of the canonically normalized Whittaker functions we define in Theorem~\ref{thm:whitnorm}.  The integral is Jacquet's integral representation  of Whittaker functions (see \cite{JacquetThesis} and \cite{Goldfeld2015}).  The product of Gamma factors  is included so that   ${W^{\pm}_{\alpha}}$ is invariant under all permutations of  $\alpha_1,\alpha_2, \ldots, \alpha_n.$  Moreover, even though Jacquet's integral  often vanishes identically as a function of $\a$, this normalization never does.
\vskip 16pt
\noindent
\begin{remark} If $g$ is a diagonal matrix in $\GL(n,\mathbb R)$ then the value of $W^{\pm}_{\alpha}(g)$ is independent of sign, so we drop the $\pm$. We also sometimes drop the $\pm$ when the sign is $+$.
\end{remark}

Jacquet's Whittaker function  is characterized (up to scalar multiples depending on $\alpha$) by the following properties:
 \begin{align} \nonumber & \bullet \; \delta W^{\pm}_{\alpha} (g) = \lambda_{\delta}(\alpha)\cdot W^{\pm}_{\alpha}(g), \qquad \big(\text{for all $\delta\in \mathcal D^n$, \, $g \in \GL(n,\R)$}\big),\\
 &\bullet \; W^{\pm}_{\alpha}(ugzk) = \psi(u) W^{\pm}_{\alpha}(g), \qquad \big(\text{for all} \; u\in U_n(\R), \; g\in \GL(n, \R),\nonumber\\
 &
 \hskip 200pt
 z\in \mathbb R^\star,\;k\in O_n(\mathbb R)\big), \nonumber\\
 &\bullet \;  (y_1\cdots y_{n-1})^N W^{\pm}_{\alpha}(\mathrm{diag}\big( y)\big)=\mathcal O(1), \quad  \big(\text{for $N>0$ and $y_i\ge 1, (1\le i<n)$}\big),  \nonumber\\
 &\bullet \; {W^{\pm}_{\alpha}}\;\text{is entire in $\alpha$,}\nonumber
 \end{align}
Furthermore, one has the functional equation
$${W^{\pm}_{\alpha}}  =  {W^{\pm}_{\alpha'}},$$  where $\alpha'$ is any permutation of $\alpha = (\alpha_1, \ldots, \alpha_n)$.

\noindent

\begin{theorem} {\bf (Fourier expansion of Maass forms)} \label{FourierExpansion}
Let $n\ge 2$ and let
$
\phi: \mathfrak h^n\to\mathbb C
$
be a Maass cusp form for $\SL(n, \mathbb Z)$
with Langlands parameters $\alpha = (\alpha_1,  \ldots,\alpha_n)\in \mathbb C^n$ (see Definition~\ref{def:MaassForm}).
 Then  for $g\in \GL(n,\mathbb R)$  \begin{align*}
\phi(g) = & \sum_{\gamma\in U_{n-1}(\mathbb Z)\backslash \SL_{n-1}(\mathbb Z)} \;\sum_{m_1=1}^\infty \;\; \cdots\;  \sum_{m_{n-2}=1}^\infty\;\sum_{m_{n-1}\ne0} \\
& \hskip 70pt
 \frac{A_\phi(M)}
{\prod\limits_{k=1}^{n-1}  |m_k|^{\frac{k(n-k)}{2}}} W^{\text{\rm sgn}(m_{n-1})}_{\alpha}\left(M^* \left(\begin{matrix} \gamma &0\\0&1\end{matrix}\right)g\right), \end{align*}
where $M^*= \left(\begin{smallmatrix} |m_1m_2\cdots m_{n-1}| & & &&\\
&\ddots & &&\\
& & |m_1m_2| &&\\
&&& |m_1|&\\
&&&&1\end{smallmatrix}\right)$.
\end{theorem}

Here $A_\phi(M)\in\mathbb C$ is called the $M^{th}$ Fourier coefficient of $\phi$. This theorem was first proved independently by  Piatetski-Shapiro \cite{PS1975} and Shalika \cite{Shalika1973,Shalika1974}. See also the proof in \cite[\S 9.1]{Goldfeld2015}.

\begin{definition}{\bf (First  coefficient of  a Maass form)} For $n\ge 2$,  consider a non-constant Maass cusp form $\phi$ for $\SL(n,\mathbb Z)$ with Fourier  expansion as in Theorem \ref{FourierExpansion}. Assume $\phi$ is a Hecke eigenform.   Let $A_\phi(1) := A_\phi(1,1,\ldots,1)$ denote the first Fourier coefficient of $\phi.$ Then we have
$$A_\phi(M) = A_\phi(1) \cdot\lambda_\phi(M).$$
Here $\lambda_\phi(M)$ is the Hecke eigenvalue (see \S 9 in \cite{Goldfeld2015}), and $\lambda_\phi(1) = 1$.
\end{definition}

Recall also the definition of the adjoint $L$-function:  $L(s, \text{\rm Ad}\;\phi) := L(s, \phi\times\overline{\phi})/\zeta(s)$ where $L(s, \phi\times\overline{\phi})$ is the Rankin-Selberg convolution $L$-function (see, for example, \cite[\S12.1]{Goldfeld2015}).

\begin{proposition} \label{PropFirstCoeff}\label{FirstCoeffMaassForm} Assume $n\ge 2.$ Let $\phi$ be a Maass cusp form for $\SL(n,\mathbb Z)$ with Langlands parameters $\alpha=(\alpha_1,\ldots,\alpha_n)$. Then the first coefficient $A_\phi(1)$ has magnitude given by
$$|A_\phi(1)|^2 =  \frac{c_n\cdot\big\langle \phi, \,\phi\big\rangle}{L(1, \text{\rm Ad} \; \phi) \prod\limits_{1\le j\ne k\le n}\Gamma\left(\frac{1+\alpha_j-\alpha_k}{2}  \right)} = \frac{c_n\cdot\big\langle \phi, \,\phi\big\rangle}{L^*(1, \text{\rm Ad} \; \phi)},$$
where $\langle \cdot,\cdot \rangle$ denotes the $L^2$-inner product over $\SL(n,\Z)\backslash \mathfrak h^n$ against some fixed choice of Haar measure and $c_n \ne 0$ is a constant depending only on $n$ and this Haar measure. Here $L^*(1, \text{\rm Ad} \; \phi)$ is the completed adjoint $L$-function (including Gamma factors).
\end{proposition}

See \cite{GSW2019} for a proof of Proposition~\ref{PropFirstCoeff}.

\section{\bf The $m^{\rm{th}}$ Hecke eigenvalue of Langlands Eisenstein series}

\begin{proposition} {\bf (The $M^{\rm{th}}$ Fourier coefficient of $E_{\mathcal P_{\text{\rm Min}}}$)} \label{LanglandsParPMin} Consider  $E_{\mathcal P_{\text{\rm Min}}}(\cdot,\alpha)$ with associated Langlands parameters $\alpha$ as defined in Definition~\ref{MinParEisSeries}. Let $M = (m_1,m_2,\ldots, m_{n-1}) \in \mathbb Z_{+}^{n-1}$ and $M^*\hskip-3pt=\hskip-2pt\text{\rm diag}(m_1\cdots m_{n-1}, \ldots, m_1, 1)$. Then the $M^{th}$ term in the Fourier-Whittaker expansion of $E_{\mathcal P_{\text{\rm Min}}}$ is
\begin{align*}
\int\limits_{U_n(\mathbb Z)\backslash U_n(\mathbb R)} E_{\mathcal P_{\text{\rm Min}}}(ug, \alpha)\, \overline{\psi_M(u) }\; du \; = \; \frac{A_{\mathcal P_{\text{\rm Min}}}(M,\alpha)}
{\prod\limits_{k=1}^{n-1}  m_k^{k(n-k)/2}} \; W_\alpha\big(M^* g\big),
\end{align*}
where $A_{\mathcal P_{\text{\rm Min}}}(M,\alpha) = A_{\mathcal P_{\text{\rm Min}}}\big((1,\ldots,1),\alpha\big) \cdot \lambda_{\mathcal P_{\text{\rm Min}}}(M,\alpha),$
and
\begin{equation}\label{HeckeEigenvaluePMin}
\lambda_{\mathcal P_{\text{\rm Min}}}\big((m,1,\ldots, 1), \alpha\big) = \underset{c_1c_2\cdots c_n=m} {\sum_{c_1,\ldots,c_n\in\mathbb Z_{+}}} c_1^{\alpha_1} c_2^{\alpha_2}\cdots  c_n^{\alpha_n}, \qquad (m\in\mathbb Z_{+})
\end{equation}
  is the $(m,1,\ldots,1)^{th}$ (or more informally, simply  the $m^{th}$) Hecke eigenvalue of $E_{\mathcal P_{\text{\rm Min}}}$. Furthermore
\begin{equation} \label{FirstCoeffPMin}
A_{\mathcal P_{\text{\rm Min}}}\big((1,\ldots,1), \alpha\big) = \, c_0 \prod_{1\leq j< k \leq n} \zeta^*\big(1 + \alpha_j - \alpha_k \big)^{-1}
\end{equation}
for some constant $c_0 \ne 0$ (depending only on $n$),   and  $\zeta^*(w)$ is the completed Riemann $\zeta$-function from (\ref{zetadef}).
\end{proposition}

\begin{proof}
The proof of (\ref{HeckeEigenvaluePMin}) follows as in the proof of Theorems 10.8.1 and 10.8.6 in \cite{Goldfeld2015}.
Equation (\ref{FirstCoeffPMin}) is a special case of Theorem~\ref{thm:templategeneral}.
\end{proof}

\vskip 5pt
The  $m^{th}$ Hecke eigenvalue of Langlands Eisenstein series twisted by Maass forms of lower rank was computed in Proposition~10.9.3 of \cite{Goldfeld2015} which we now rewrite (in the notation of Definition~\ref{EisensteinSeries}) for the convenience of the reader. Note that what we write here is a corrected version which can be found in the errata to \cite{Goldfeld2015}.  Note also that we keep the above notation using $(\alpha_1,\ldots,\alpha_n)$ for parameterizing minimal parabolic Eisenstein series, but use the $s$-variables from Definition~\ref{EisensteinSeries} for other parabolics.

\begin{proposition} {\bf (The $m^{\rm{th}}$ Hecke eigenvalue of $E_{\mathcal P_{n_1,\ldots,n_r}}$)} \label{Prop10.9.3} Let $n \ge 2$ and $n = n_1+\cdots n_r$ be a partition of $n$ with $1 < r < n.$  Let $E_{\mathcal P, \Phi}(g,s)$, with $\mathcal P = \mathcal P_{n_1,\ldots,n_r}$,  $\Phi = (\phi_1, \ldots,\phi_r)$, be a Langlands Eisenstein series for $\SL(n, \mathbb Z)$ where  $s = (s_1, \ldots,s_r) \in \mathbb C^r$ with $n_1s_1+\cdots+n_rs_r = 0.$  Let $T_m$ denote the Hecke operator on $\SL(n,\mathbb Z)$  for $m = 1, 2, 3, \ldots$ Then $T_{m}  E_{\mathcal P, \Phi} \; = \;\lambda_{E_{P,\Phi}}\big((m,1,\ldots,1),s\big) \cdot  E_{\mathcal P, \Phi}$ where
\begin{align*}
    \lambda_{E_{P,\Phi}}\big((m,1,\ldots,1),s\big) & = \hskip-10pt\underset {c_1c_2\cdots c_r = m } {\sum_{    1 \le c_1, c_2, \ldots, c_r \, \in \, \mathbb Z}} \hskip-10pt \lambda_{\phi_1}(c_1)  \cdots \lambda_{\phi_r}(c_r)
    \cdot c_1^{s_1-\rho_{_{\mathcal P}}(1)}  \cdots c_r^{s_r-\rho_{_{\mathcal P}}(r)}
\end{align*}
     (cf.~(\ref{rhoPj})).
      In this formula $\lambda_{\phi_i}(c_i)$ denotes the eigenvalue of the $\SL(n_i, \mathbb Z)$  Hecke operator $T_{c_i}$ acting on $\phi_i$ which may also be viewed as  the $(c_i,1,1,\ldots,1)^\text{th}$ Fourier coefficient of $\phi_i.$
\end{proposition}

\subsection*{\bf The $\SL(4,\Z)$ Langlands Eisenstein series}

There are 4 standard, proper, non-associate (i.e., with non-conjugate Levi components $M^\mathcal P$) parabolic subgroups $\mathcal P$ of $\SL_4(\mathbb R)$, corresponding to the partitions
\begin{align} \label{GL(4,R)partitions}
4 \; = \; 3+1 \; =  \; 2+2 \;   = \; 2+1+1 \;  = \; 1+1+1+1.
\end{align}
We now explicitly compute the Fourier coefficients of the Langlands Eisenstein series for
$\SL(4, \mathbb Z).$  Let $E_{\mathcal P, \Phi}(\cdot,s)$  have Langlands parameters $\alpha = (\alpha_1,\alpha_2,\alpha_3,\alpha_4)$ which depend on $s$ and the Langlands parameters of $\Phi$ as in \eqref{langlandsparamsforPhi}. Now $E_{\mathcal P, \Phi}$  will have a Fourier-Whittaker expansion similar to  the one in Theorem~\ref{FourierExpansion}.  In particular, if $M = (m_1,m_2,m_3) \in \mathbb Z^3$ with $m_1,m_2,m_3\ge 1,$  then the $M^{th}$ Fourier-Whittaker coefficient of $E_{\mathcal P, \Phi}(\cdot, s)$ is given by
\begin{align*}
\int\limits_{U_4(\mathbb Z)\backslash U_4(\mathbb R)} E_{\mathcal P}(uy, s)\, \overline{\psi_M(u) }\; du \; = \; \frac{A_{E_{\mathcal P, \Phi}}(M, s)}
{m_1^{\frac32} m_2^2 m_3^{\frac32}} \; W_\alpha\big(M^* y\big).
\end{align*}
Here $$A_{E_{\mathcal P, \Phi}}(M, s) = A_{E_{\mathcal P, \Phi}}\big((1,1,1),s\big) \cdot \lambda_{E_{\mathcal P, \Phi}}(M, s)$$
and  $\lambda_{E_{\mathcal P, \Phi}}(M)$
  is the $M^{th}$ Hecke eigenvalue of $E_{\mathcal P, \Phi}$.

The above definition of Eisenstein series uses the parameters $s=(s_1,\ldots,s_r)$ from \cite{Goldfeld2015}.  However, many of our formulas simplify if one writes
\begin{equation}\label{sizi}
  s_i=z_i+\rho_{_{\mathcal P}}(i), \qquad (z_i\in\C).
\end{equation}
Note that $\sum\limits_{i=1}^r z_i=0$ since $\sum\limits_{i=1}^r s_i=0.$

\vskip 6pt
 It follows from Proposition \ref{Prop10.9.3} that the $(m,1,\ldots,1)^{\text{th}}$ Hecke eigenvalue of the  Langlands Eisenstein series associated to the 4 standard partitions are (respectively) the four divisor sums
 \begin{align*}
 & \sum\limits_{c_1c_2c_3c_4=m} \hskip-15pt c_1^{\alpha_1} c_2^{\alpha_2} c_3^{\alpha_3} c_4^{\alpha_4},\qquad \hskip 52pt
\sum\limits_{c_1c_2c_3 =m} \lambda_\phi(c_1) c_1^{z_1} c_2^{z_2}c_3^{-2z_1-z_2},\\
&
\sum\limits_{c_1c_2=m}\hskip-5pt \lambda_{\phi_1}(c_1)  \lambda_{\phi_2}(c_2) \left(\frac{c_1}{c_2}\right)^{z_1}, \qquad \sum\limits_{c_1c_2=m}\hskip-5pt \lambda_{\phi_1}(c_1)  \lambda_{\phi_2}(c_2) \left(\frac{c_1}{c_2^3}\right)^{z_1},
\end{align*}
which appear in the Table of Fourier Coefficients of $\SL(4,\mathbb Z)$ Langlands Eisenstein Series found at the end of Part I.

\section{\bf Sketch of the template method for $\SL(3,\Z)$ and $\SL(4,\Z)$ Langlands Eisenstein series}\label{sec:sketch}

The template method can be used to easily and quickly compute the first Fourier coefficient of Langlands Eisenstein series for $\GL(n)$ up to a constant factor.  This is a special case of the more precise relationship (that gives the correct constant factor) given later in Theorem~\ref{thm:templategeneral} for general Chevalley groups.
 For minimal parabolic Eisenstein series the Fourier coefficient is described by \eqref{FirstCoeffPMin} in Proposition~\ref{LanglandsParPMin}.  The Langlands Eisenstein series twisted by minimal parabolic Eisenstein series of lower rank are in fact also minimal parabolic Eisenstein series, so their first coefficient can be easily found  as above, and (up to a constant factor) the first coefficient will be a product of completed $\zeta$-functions associated to the minimal parabolic Eisenstein series of lower rank.  According to the template method, one can then simply replace the completed $\zeta$-functions associated to the minimal parabolic Eisenstein series of lower rank by completed $L$-functions associated to Maass forms of lower rank. There is one other technical issue that comes up, however:~normalization.

We shall require the following standard notation for completed $L$-functions. Recall the completed Riemann $\zeta$-function from (\ref{zetadef}).
For a Maass form $\phi$ on $\GL(2)$  with Langlands parameters $(v,-v)$ (i.e., spectral parameter $1/2+v$), define the completed $L$-function $L^*(w, \phi)$  associated to $\phi$  (for $w\in\mathbb C$)
by
$$L^*(w, \phi) :=  \pi^{-w}\Gamma\left( \frac{w+v}{2}  \right)  \Gamma\left( \frac{w-v}{2}  \right) L(w,\phi) = L^*(1-w,\phi).$$
If $\phi_1, \phi_2$ are two Maass forms on $\GL(2)$ with Langlands parameters $(v,-v)$ and $(v',-v')$, respectively, then the completed $L$-function for the Rankin-Selberg convolution $L(w, \phi_1\times\phi_2)$ is given by
\begin{align} \label{RankinSelbergL}
L^*(w, \phi_1\times\phi_2) & =   \pi^{-2w}\Gamma\left( \frac{w+v+v'}{2}  \right)  \Gamma\left( \frac{w-v+v'}{2}  \right)\\
&
\hskip 30pt
\cdot
\Gamma\left( \frac{w+v-v'}{2}  \right)  \Gamma\left( \frac{w-v-v'}{2}  \right) L(w,\phi_1 \times \phi_2).
\nonumber
\end{align}
Finally, for a Maass form $\phi$ on $\GL(3)$ with Langlands parameters $(\alpha_1,\alpha_2,\alpha_3)=(v+2v',v-v',-2v-v')$ (i.e., spectral parameters $(\frac13 +v, \frac13 +v')$) define the completed $L$-function
$L^*(w, \phi)$  associated to $\phi$ by
\begin{align*}
L^*(w, \phi) :=  \pi^{-\frac{3w}{2}}G_{v,v'}(w) L(w,\phi) = L^*(1-w,\phi),
\end{align*}
where
$$
\aligned
G_{v,v'}(w) & = \prod_{j=1}^3\Gamma\left(\frac{w+\alpha_j}2\right)\\
&  = \Gamma\left( \frac{w+v+2v'}{2}  \right)  \Gamma\left( \frac{w+v-v'}{2}  \right) \Gamma\left( \frac{w-2v-v'}{2}  \right).
\endaligned$$
Recall also that the adjoint $L$-function of a Maass form $\phi$ on $\GL(n)$ over $\Q$ is defined by $L(w, \text{\rm Ad} \; \phi) := L(w,\phi\times \overline{\phi})/\zeta(w).$

\subsection*{\bf Computing the first coefficient of  $E_{\mathcal P_{2,1}, \Phi}$:}
We begin by explicitly working out an example of the template method for $\SL(3,\mathbb Z)$ which sketches a proof of part 2) of Theorem \ref{thm:inintro}. Consider the Eisenstein series $E_{P_{2,1},\Phi}$ where $\Phi$ is a Maass form for $\SL(2, \mathbb Z)$ with Langlands parameters ($v, -v)$, Laplace eigenvalue $\frac14-v^2,$ and spectral parameter $\frac12+v.$
\vskip 5pt
{\bf Step 1:} We replace $\Phi$ with the Borel Eisenstein series $  E_{P_{\rm Min}}\Big(\left(\begin{smallmatrix}  y_2 &\\
& 1\end{smallmatrix}\right), \,v \Big) $ for $\SL(2,\mathbb Z)$, which has  the same spectral parameter as $\Phi$.
We then define an Eisenstein series of ${\mathcal P}_{2,1}$ using Definitions~\ref{def:maassparabolic} and \ref{EisensteinSeries}, but with $\Phi$ replaced by $  E_{P_{\rm Min}}\Big(\left(\begin{smallmatrix}  y_2 &\\
& 1\end{smallmatrix}\right), \,v \Big) $.

In more detail, recall the slash notation $F(g)\Big|_\gamma := F(\gamma g)$ on functions $F:\frak h^n\to\mathbb C$, in which  $n\ge 2$ and $\gamma\in \GL(n,\mathbb R)$.
 Then we form the new Eisenstein series
\begin{align*}
E_{\mathcal B, \text{new}}(y,s) & := \sum_{\gamma\in P_{2,1}(\mathbb Z)\backslash \SL(3,\mathbb Z)} E_{P_{\text{\rm Min}}}\Big(\left(\begin{smallmatrix}  y_2  &\\
& 1\end{smallmatrix}\right), \,\nu \Big) \Big(y_1^2 y_2\Big)^{\frac12+z_1}\bigg|_{\displaystyle \gamma}
\\
& =    \sum_{\gamma\in U_3(\mathbb Z)\backslash \SL(3,\mathbb Z)} y_2^{\frac12+\nu} \left(y_1^2 y_2   \right)^{\frac12+z_1}\bigg|_{\displaystyle \gamma},
\end{align*}
which is a minimal parabolic Eisenstein series for $\SL(3,\mathbb Z).$
By   (\ref{langlandsparamsforPhi}) and (\ref{sizi}), we see $E_{\mathcal B, \text{new}}$ has  Langlands parameters $(z_1+v,z_1-v,-2z_1)$.

\vskip 5pt
{\bf Step 2:}  It then follows from (\ref{FirstCoeffPMin}) that the first coefficient of $E_{\mathcal B, \text{new}}$ is given by the formula
$$\zeta^*(2v) A_{\mathcal B, \text{new}}\big((1,1,1),s\big)  = \Big(\zeta^*(1+3z_1+v) \zeta^*(1+3z_1-v)   \Big)^{-1},$$
where the factor of $\zeta^*(2v)$ on the left corresponds to the factor in (\ref{FirstCoeffPMin}) corresponding to the (1,2) entry of the upper $2\times 2$ block of the Levi of $\mathcal P_{2,1}$.  Note also that by Proposition \ref{LanglandsParPMin}, $ \zeta^*(2v) E_{P_{\rm Min}}\Big(\left(\begin{smallmatrix}  y_2 &\\
& 1\end{smallmatrix}\right), \,v \Big) $  is  normalized to have its first Fourier-Whittaker coefficient equal to one (up to a constant factor), so it plays a  role parallel to a cusp form normalized the same way.
\vskip 5pt
{\bf Step 3:}
It follows from the template method that (up to normalization) the Eisenstein series $E_{P_{2,1},\Phi}(\cdot, s)$ must have first coefficient equal to $L(1+3z_1,\phi)$.  To settle the normalization, if  $\Phi$ is a $\GL(2,\mathbb R)$ Maass form with Petersson norm $\langle\Phi, \Phi\rangle =1$, then $\Phi/A_{\Phi}(1)$ will have first Fourier coefficient equal to  an absolute constant. It follows from Proposition \ref{PropFirstCoeff} that
$$A_\phi(1)=  L^*(1, \text{Ad} \; \phi)^{-\frac12}$$
up to a constant factor. Since an Eisenstein series is generally not in $L^2$, the only canonical way to normalize it is by having its first coefficient be equal to one. This recovers part 2) of Theorem~\ref{thm:inintro}.

\vskip 8pt
We now illustrate the method in the case of $n=4$ by computing the first coefficient of the  Eisenstein series $E_{\mathcal P_{2,2},\Phi}$.  The analogous computations for the other parabolic subgroups are omitted as they are very similar.

\subsection*{\bf Computing the first coefficient of  $E_{\mathcal P_{2,2}, \Phi}$:}

Here in addition to $E_1:=E_{P_{\rm Min}}\Big(\cdot , \,v \Big) $  we take a second Eisenstein series $ E_2:= E_{P_{\rm Min}}\Big(\cdot, \,v' \Big) $ with spectral parameter matching $\phi_2$.
The completed $L$-functions associated to $\Phi^*$ are
$$L^*(w, E_1) =  \zeta^*(w+v) \zeta^*(w-v) \quad \text{and} \quad L^*(w, E_2) =  \zeta^*(w+v') \zeta^*(w-v').$$
We see from (\ref{langlandsparamsforPhi})  and (\ref{sizi}) that the Langlands parameters of
$$E_{P_{2, 2},\Phi^*}\big((1,1,1),s\big), \qquad \text{with~}s=(z_1+1,-z_1-1),$$ equal $(z_1+v,z_1-v,-z_1+v',-z_1-v')$.
It  follows from  (\ref{RankinSelbergL}) and Proposition \ref{LanglandsParPMin}
that
\begin{align*}
&
\zeta^*(1+2v)\zeta^*(1+2v')A_{E_{P_{2, 2},\Phi^*}} \big((1,1,1),s\big)    = \\
   &
\hskip40pt = \Big(\zeta^*(1 + 2z_1  - v - v')
\zeta^*(1 + 2z_1  + v - v')\\
&
\hskip 100pt
\cdot   \zeta^*(1 + 2 z_1 - v + v')
 \zeta^*(
  1 + 2 z_1  + v + v')\Big)^{-1} \\
  &
  \hskip 40pt
  = L^*\big(1+2z_1,E_1\times E_2\big)^{-1}.
  \end{align*}
  According to the template method, the same result must hold if $E_1, E_2$ are replaced by Maass forms, normalized to have first Fourier-Whittaker coefficient equal to one (up to a constant factor).  Now, if $\phi$ is a $\GL(2,\mathbb R)$ Maass form with Petersson norm $\langle\phi, \phi\rangle =1$, then $\phi/A_{\phi}(1)$ will have first coefficient equal to one (up to a constant factor). Further, if $\phi$ has spectral parameter $\frac12+\mu$, it follows from Proposition \ref{PropFirstCoeff} that
$$A_\phi(1)= \bigg(\Gamma\left(\frac12+\mu\right) \Gamma\left(\frac12-\mu\right) L(1, \text{Ad} \; \phi)\bigg)^{-\frac12}$$
up to a constant factor. Hence,  we may replace the minimal parabolic Eisenstein series $E_1, E_2$ by
Maass forms
 \[ \left|\Gamma\Big(\frac12+v\Big)\right|\sqrt{ L(1, \text{\rm Ad}\;\phi_1)}\cdot \phi_1 \quad\mbox{and} \quad \left|\Gamma\Big(\frac12+v'\Big)\right|\sqrt{L(1, \text{\rm Ad}\;\phi_2)} \cdot \phi_2,\]
respectively.
It follows that the first coefficient of   $E_{\mathcal P_{2,2}, \Phi}$, where $\phi_1,\phi_2$ are Maass forms of norm 1 on $\GL(2)$ with spectral parameters $\frac12+v, \frac12+v'$, is given by
\begin{align*}
&
\Bigg(L(1, \text{\rm Ad}\; \phi_1)^{\frac12} \, L(1, \text{\rm Ad}\; \phi_2)^{\frac12}
  \left| \Gamma\left(\frac12+v  \right)\; \Gamma\left(\frac12+v' \right)  L^*\big(1+2z_1,\phi_1\times\phi_2\big)\right|\Bigg)^{-1},
  \end{align*}
  up to multiplication by an absolute constant.
\qed

\vskip 15pt

We record the first coefficient and the $m^{th}$ Hecke eigenvalue entries for each Eisenstein series in the following table.  The entry for the minimal parabolic Eisenstein series with partition 4 =1+1+1+1 follows from Proposition~\ref{LanglandsParPMin}, while the rest are obtained from   Proposition~\ref{Prop10.9.3}.

The following table lists, for each partition, the Maass form $\Phi$ (with its associated spectral parameters), the values of $s$-variables, Langlands parameters, and the Fourier-Whittaker coefficients of the $\SL(4, \mathbb Z)$ Eisenstein series $E_{\mathcal P,\Phi}$. Note that $\mathcal P_{1,1,1,1}= \mathcal P_{\rm Min} =\mathcal B.$

The first coefficient entries in the table of Fourier coefficients of $\SL(4,\mathbb Z)$ Eisenstein series are well known in much greater generality (see, for example, Theorem~7.1.2 of \cite{Shahidi2010}).  This along with an analysis of the generic Fourier coefficients of spherical cusp forms is the subject of Part II of this paper.

\begin{center}
\textbf{FOURIER COEFFICIENTS OF $\SL(4,\mathbb Z)$ LANGLANDS EISENSTEIN SERIES}
\vskip 6pt
\nopagebreak
\includegraphics[scale=.76]{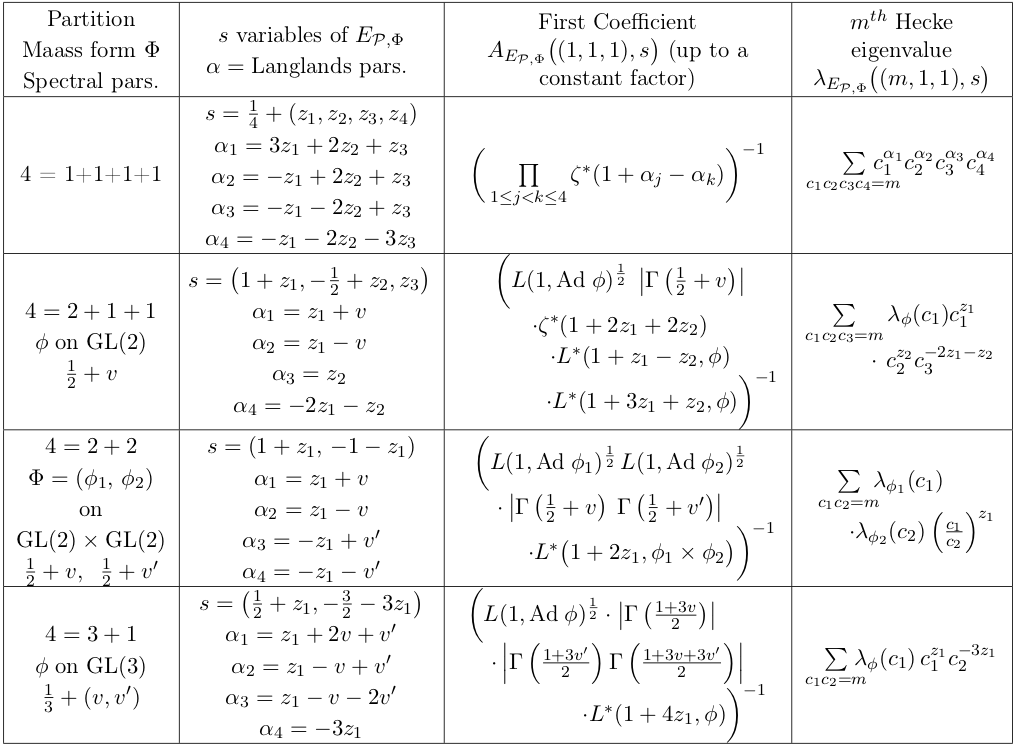}
\end{center}


\begin{remark}
The formulas given here for the first coefficient are valid when the inducing function on $\GL(n_i)$, for $n_i>1$, is a Maass cusp form.  Analogous but somewhat different results can be shown when the inducing function is a constant function.
\end{remark}

\part{Template method for Chevalley groups}

Having described the main results in classical terms for $\GL(n)$ in Part I, we now turn to general Chevalley groups.  Chevalley groups are particular realizations of split semisimple Lie groups endowed with a coordinate system satisfying certain integrality properties.  This algebraic structure allows their definition to be extended to arbitrary fields.  Also, Chevalley groups can be written in terms of generators and relations involving $\SL(2)$-subgroups.  A standard reference is Steinberg's Yale monograph \cite{Steinberg}, though some of the terminology used there has changed over the years.  For a more recent account close to ours, see \cite[\S2]{HundleyMiller}.

\section{Introduction to roots and powers in the elementary example of the Chevalley group $\SL(3)$}\label{sec:Chevalley1}

The template method works equally well for general Chevalley groups as it does for $SL(n)$ or  $PGL(n)$. To illustrate our methods and results we  begin with the special case of the Chevalley group $\SL(3)$ of rank 2, and illustrate some of the basic concepts over the field $\R$. Its associated special linear Lie algebra $\pmb{\frak s} \hskip 0.5pt \pmb l_3(\R)$ is the algebra of real $3\times 3$ matrices with trace zero and Lie bracket given by
$[X,Y]:= XY - YX$ with $X,Y\in \pmb{\frak s } \hskip 0.5pt \pmb l_3(\R).$

\begin{definition} {\bf (Cartan subalgebra)}
A Cartan subalgebra $\frak a$ of $\pmb{\frak s} \hskip 0.5pt \pmb l_3(\R)$  is
a maximal abelian subalgebra consisting of diagonalizable matrices in $\pmb{\frak s} \hskip 0.5pt \pmb l_3(\R)$.
\end{definition}
\begin{remark} The abelian condition simply means $[X,Y]=0$ for all $X,Y\in\frak a$.  The subalgebra of  all diagonal matrices $H=\diag(h_1,h_2,h_3)$ having $\tr(H)=h_1+h_2+h_3=0$ is a Cartan subalgebra, and it is a general fact all other Cartan subalgebras of  $\pmb{\frak s} \hskip 0.5pt \pmb l_3(\R)$ are $\SL(3,\R)$-conjugate to $\frak a$.
\end{remark}

 For all $X\in \pmb{\frak s } \hskip 0.5pt \pmb l_3(\R)$ we define the linear map $\text{ad}(X): \pmb{\frak s } \hskip 0.5pt \pmb l_3(\R) \to \pmb{\frak s } \hskip 0.5pt \pmb l_3(\R)  $ given by $\text{ad}(X)(Y) := [X,Y].$ For $H\in \frak a$, an eigenvector of $\text{ad}(H)$ is   an element $X\in  \pmb{\frak s } \hskip 0.5pt \pmb l_3(\R)$ such that
 $$\text{ad}(H)(X) = \lambda \cdot X$$
for some eigenvalue $\lambda \in \mathbb C.$  Define $E_{i,j}$ to be the elementary $3\times 3$ matrix with a one at the $(i,j)$-th position and zeros elsewhere. A simultaneous eigenbasis for $\pmb{\frak s } \hskip 0.5pt \pmb l_3(\R) $ is given by  the six elementary matrices $E_{i,j}$ with  $1\le i,j\le 3$ and $i \ne j$, along with any basis of $\frak a.$ For $H = \text{diag}(h_1,h_2,h_3) \in \frak a$ we compute
$$\text{ad}(H)(E_{i,j}) = (h_i-h_j) \cdot E_{i,j}.$$
Since $[H,H'] = 0$ for all $H,H'\in\frak a$ we see that $\text{ad}(H)(H')=0.$ The linear maps $H \to h_i-h_j$ (for $i\neq j$) are termed roots:

\begin{definition} {\bf (Roots associated to the Cartan subalgebra $\frak a$)}\label{def:6.2} Let $H = \text{\rm diag}(h_1,h_2,h_3) \in \frak a$.
The simple roots $\Sigma := \{\a_1,\,\a_2\}$ are linear functionals on $\frak a$, defined by the formula
$$\a_j(H) := h_j-h_{j+1}, \qquad\quad  (1\le j \le 2).$$
The positive roots $\D_+$ are the linear functionals which map $H \in \frak a$ to $h_i-h_j, \;(1\le i<j \le 3)$. It is easy to see that $$\Delta_+ := \{\a_1,\a_1+\a_2, \a_2   \}.$$ The negative roots are $\Delta_- := - \Delta_+$ and the set of all roots $\Delta$  is the disjoint union of $\Delta_+$ and $\Delta_-$.
\end{definition}

\begin{definition}\label{def:dualcartan}{\bf (The complex dual of the Cartan)}
We also define $\frak a_{\mathbb C}^* := \mathbb C \alpha_1 \oplus \mathbb C \alpha_2$, and for $
\lambda = s_1 \alpha_1 + s_2\alpha_2 \in \frak a_{\mathbb C}^*$ with $s_1,s_2\in\mathbb C$ we define
$$\lambda(H) := s_1 \alpha_1(H) +s_2\alpha_2(H), \qquad (\text{for all} \; H\in\frak a).$$
\end{definition}

\begin{definition}\label{def:rootspaceXalpha} {\bf (The root space ${\mathcal X}_\alpha$ and root vectors $X_\alpha$)} Let $\alpha\in \Delta.$ We define the root space ${\mathcal X}_\alpha := \big\{ X\in \pmb{\frak s } \hskip 0.5pt \pmb l_3(\R) \; \big | \; [H,X] = \alpha(H)\cdot X \;\text{for all} \;H\in\frak a\big\}.$ It is always the case that  $\text{\rm dim}\; {\mathcal X}_\alpha =1$, so we usually instead consider a fixed nonzero element $X_\a\in {\mathcal X}_\a\subset  \pmb{\frak s } \hskip 0.5pt \pmb l_3(\R)$, called a root vector.  The notion of Chevalley group introduces a scaling constraint on which root vector in $\mathcal X_\alpha$ to choose.
\end{definition}

\begin{definition} {\bf (One-parameter subgroups associated to $\a$)} Given $\alpha\in \Delta$ and $X_\a$ as in Definition~\ref{def:rootspaceXalpha}, the matrix exponentials $u_\alpha(t)=\exp(t X_\a)$  ($t\in\R$) define a one-parameter subgroup of $\SL(3,\R)$.  Since $X_\alpha$ is nilpotent, this exponential is actually a finite sum; in particular, the matrix entries of $u_\alpha(t)$ are  polynomials in $t$.
\end{definition}

\begin{definition} {\bf (Characters of exponentials on $\frak a$)}\label{def:charsofexpsona}
Let  $H\in\frak a.$ For a general element $ \lambda = s_1 \alpha_1 + s_2\alpha_2 \in \frak a_{\mathbb C}^*$ we define
$$\exp(H)^\lambda := \exp(\lambda(H)) = \exp(\alpha_1(H))^{s_1} \cdot \exp(\alpha_2(H))^{s_2}.$$
\end{definition}

In the special case that $\lambda=\alpha\in\Delta$ is itself a root, this definition gives an extension of the root notation  (which first appears above as a linear functional on $\frak a$) to   a homomorphism of the image of the exponential map on $\frak a$.  For example, if $H=\operatorname{diag}(\log (t_1),\log (t_2),\log (t_3))\in\frak a$ and $t = \operatorname{diag}(t_1,t_2,t_3)=\exp(H)$, then $t^{\a_1}=t_1/t_2$ and $t^{\a_2}=t_2/t_3$.  These last two maps are algebraic homomorphisms, in that they are defined by rational functions on diagonal matrices (not just ones with positive entries as in the following definition).

\begin{definition} {\bf (The power function $t \to t^\lambda$)}\label{def:charsont}
Let  $H\in\frak a$ and set $t = \exp(H)$ where
$H=\operatorname{diag}(\log (t_1),\log (t_2),\log (t_3))$ and
$t=\operatorname{diag}(t_1,t_2,t_3)$ with $t_i>0$.  Then
$$\alpha_1(H) \ \ = \ \ \log(t_1)-\log(t_2), \qquad \alpha_2(H) \ \ = \ \ \log(t_2)-\log(t_3)$$
(see Definition~\ref{def:6.2}).
It follows that for $\l=s_1\alpha_1+s_2\alpha_2$ as in Definition \ref{def:charsofexpsona},
$\;\boxed{t^\l=\exp(H)^\lambda=t_1^{s_1}t_2^{s_2-s_1}t_3^{-s_2}.}$
\end{definition}

The power function $t\to t^\lambda$ in the above definition can be used to construct the eigenfunctions $I_s$ of definition \ref{def:conversion} in the more general setting of Chevalley groups. The constructions of Whittaker functions and Eisenstein series occurring in part 1 (the template method for $\GL(n)$) of this paper are given by certain integrals and sums of the power functions $I_s$, respectively. We now have the tools to generalize these constructions to Chevalley groups in the adelic setting which allow us to rederive Shahidi's  beautiful formulas   for the Fourier coefficients of Eisenstein series for Chevalley groups that form the theoretical basis of the template method.

\section{Notation and background on Chevalley groups}\label{sec:Chevalley2}

Let $G$ be a Chevalley group with root system $\Delta\subset \R^r$, where $r$ is the rank of $G$.  We fix a choice of positive roots $\Delta_+\subset \Delta$, so that $\Delta$ is the disjoint union of $\Delta_+$ and the negative roots $\Delta_{-}=-\Delta_+$.  Let  $\rho$ denote the half-sum of all positive roots and let $\Sigma=\{\alpha_1,\ldots,\alpha_r\}\subset \Delta_+$ denote the simple roots.
 The root system   comes equipped with a positive-definite inner product $(\cdot,\cdot)$, the restriction of the euclidean inner product on $\R^r$ to $\Delta$.  This allows us to identify the $\R$-span of the root system with its linear dual.  Given a root $\alpha$,
we use the notation $\alpha^\vee=\frac{2}{(\a,\a)}\a$ for the coroot of $\alpha\in\Delta$.  For many purposes it is convenient to think of coroots as dual to roots using the identification via $(\cdot,\cdot)$, and to write $\langle \alpha,\beta^\vee\rangle$ for $(\a,\b^\vee)=2\frac{(\a,\b)}{(\b,\b)}$.

The fundamental weights $\varpi_{\alpha_1},\ldots,\varpi_{\alpha_r}$ are the basis of the $\Q$-span of $\Delta$ dual to the simple coroots under this pairing, i.e., they satisfy the relation
\begin{equation}\label{fundweightdef}
  \langle \varpi_{\alpha_i},\alpha_j^\vee\rangle \ \ = \ \ \delta_{i=j} \ \ = \ \ \left\{
                                                                                    \begin{array}{ll}
                                                                                      1, & i=j, \\
                                                                                      0, & i\neq j.
                                                                                    \end{array}
                                                                                  \right.\end{equation}
Sometimes we write $\varpi_i$ as shorthand for $\varpi_{\a_i}$.

By definition, the Chevalley group comes equipped with a special basis of the Lie algebra of $G$, on which $G$ acts algebraically by the adjoint action.  In more detail, this ``Chevalley''  basis  $\{X_\alpha|\alpha\in\Delta\}\cup\{H_\alpha|\alpha\in\Sigma\}$
contains root vectors $X_\alpha$ for each $\alpha\in \Delta$,  where $H_\alpha=[X_\alpha,X_{-\alpha}]$.
 The integral span of   the Chevalley basis is termed the ``Chevalley lattice''.

Each $X_\a$, for $\alpha\in \Delta$, exponentiates to an algebraic subgroup $\{u_\alpha(\cdot)\}$ isomorphic to $\mathbb G_a$, subgroups which (like $G$ itself) are defined over any field.

\begin{definition}\label{def:N} {\bf (Maximal unipotent subgroup $N$) }
Let $N$ denote the algebraic subgroup of $G$ containing all the  one-parameter subgroups  $\{u_\alpha(\cdot)\}$, $\alpha \in \Delta_{+}$.
\end{definition}

Products of unipotent elements define the elements

\begin{equation}\label{halpha}
  w_\alpha(t) \ \ = \ \ u_\alpha(t)u_{-\alpha}(-t\i)u_\alpha(t)\ \text{and} \
  h_\alpha(t) \ \ = \ \ w_\alpha(t)w_\alpha(1)\i\,, \ (t\neq 0)\,,
\end{equation}
over any field.  The maps $h_{\alpha_j}(\cdot)$, $1\le j\le r$, are naturally associated with the simple coroots $\alpha_j^\vee$, as are the Chevalley basis vectors $H_{\a_j}$.  Each $h_{\alpha_j}$ maps $\mathbb{G}_m$ to an abelian algebraic subgroup of $G$.

\begin{definition}\label{def:T}{\bf (Maximal torus $T$)}
Let $T$ denote the smallest algebraic subgroup of $G$ containing the images of each $h_{\alpha_j}(\cdot)$, $1\le j \le r$.
\end{definition}

The elements $w_\alpha(1)$ all lie in $G(F)$ for any field $F$, and
 the group generated by $\{w_\alpha(1)|\alpha\in \Delta\}$ contains a full set of representatives for the Weyl group $W$ of $\Delta$; we will tacitly identify Weyl group elements with  any such fixed set of representatives.

 \begin{definition}{ \bf (Borel subgroup)} \label{def:borel} The Borel subgroup of  the group  $G$  is $B=TN=NT$,  the semidirect product of $N$ and $T$.  \end{definition}

\noindent
{\bf Example:} On $\SL(n)$ or $\rm{PGL}(n)$, we have
 $$T = \left\{ \left(\begin{smallmatrix} \displaystyle * & &\\
& \ddots &\\
& &\displaystyle * \end{smallmatrix}\right)\right\}, \qquad N = \left\{\left(\begin{smallmatrix} \displaystyle 1 & &\phantom{.}_{_{_{\displaystyle *}}}\\
&\; \ddots &\\
& &\displaystyle 1 \end{smallmatrix}\right)\right\}, \quad \text{and} \
\quad B =  \left\{\left(\begin{smallmatrix} \displaystyle * & &\phantom{.}_{_{_{\displaystyle *}}}\\
&\; \ddots &\\
& &\displaystyle * \end{smallmatrix}\right)\right\}.$$

\begin{definition}{ \bf{(Parabolic subgroups)}}\label{def:parabolic}
  A  standard parabolic subgroup is any subgroup containing the Borel subgroup, and a parabolic subgroup is any conjugate of a standard parabolic subgroup.  For example, the Borel  subgroup is a minimal parabolic subgroup. A standard parabolic  has a decomposition of the form $P=LU$, where $U\subset N$ is its unipotent radical and $L$ is a reductive Levi component.  In general $L$ in this decomposition is determined only up to conjugation by $U$.  In the context of Chevalley groups, there is a particularly convenient choice of $L$ (which we call a ``standard Levi'') containing $T$ and the derived subgroup $L_{der}$, where $L_{der}$ is itself a Chevalley group with root system $\Delta_L$ containing all of the one-parameter subgroups $\{u_\alpha(\cdot)\}$ ($\alpha\in \Delta_L$).
\end{definition}

\vskip 1pt
\noindent
\begin{remark} See Definition \ref{GLnParabolic} for examples of standard parabolic subgroups of $\GL(n).$
\end{remark}

\vskip 10pt

Let $P=LU$ denote a standard parabolic with $L$ a standard Levi.  Let $\Delta_U=\Delta_{+}-\Delta_L$ denote the  set of roots $\a$ whose one-parameter subgroups are contained in $U$.
Let $W_L\subset W$ denote the Weyl group of $L$.  Let $w_{\rm{long}}\in W$ denote the long Weyl group element in $W$, and $w_{\rm{long},L}$ denote the long Weyl group element in $W_L$.
Conjugation by $w_{\rm{long}}$ maps $L$ to the isomorphic standard Levi subgroup $L'=w_{\rm{long}}Lw_{\rm{long}}\i$ of the (possibly same) standard parabolic  $P'=L'U'$, but note that $w_{\rm{long}}Pw_{\rm{long}}\i$ is not a standard parabolic; instead it contains
 $$B_{-}=w_{\rm{long}}\i \, B \, w_{\rm{long}},$$   the  parabolic opposite to $B$ (that is, $B_{-}=TN_{-}=N_{-}T$, where $N_-$ contains all one-parameter subgroups $\{u_\alpha(\cdot)\}$ for $\alpha\in\Delta_{-}$).

\begin{definition}\label{def:Bruhat} {\bf (Bruhat decomposition)}
For any field $F$, the group $G(F)$ satisfies the Bruhat decomposition $G(F)=B(F)\cdot W \cdot N(F)$, where as mentioned above the Weyl group $W$ is identified with a set of representatives in $G(F)$.
\end{definition}

 For any standard parabolic $P$ one has the disjoint union
\begin{equation}\label{bruhatdisjoint}
  G(F) \ \ = \ \ \bigsqcup_{w\in W_L\backslash W} P(F)wN_w(F),
\end{equation}
 where $N_w=(w\i U_{-}w)\cap N$, and $U_{-}$ is the unipotent subgroup containing all one parameter subgroups for $\a\in\D_{-}-\D_L$.  We also have the unipotent groups $N^w=(w\i Pw)\cap N$, $w\in W$.

\subsection*{Chevalley groups over the adeles}

 By definition, the subgroup of  $G(\Q)$ which (setwise) maps Chevalley lattice to itself is called the ``integral points'' $G(\Z)$  of $G(\Q)$.\footnote{Note that this notion of integrality is independent of the realization of the algebraic group $G$ in matrices.}  It has local analogs $G(\Z_p)$, defined as the stabilizer in $G(\Q_p)$ of the  Chevalley lattice tensored with $\Z_p$.
Since the one-parameter unipotent subgroups $\{u_\alpha(\cdot)\}$ are isomorphic to ${\mathbb G}_a$, their $\R$-points have natural Haar measures assigning volume 1 to  $\{u_\alpha(t) \mid 0\le t \le 1\}$, and their $\Q_p$ points similarly have natural Haar measures assigning volume 1 to  $\{u_\alpha(t) \mid   t\in\Z_p\}$.   Taking the product over all places $v\le\infty$ gives an adelic Haar measure for which the image of $\A/\Q$ has volume 1.  Furthermore, taking products of these adelic Haar measures over these various one-parameter subgroups  gives Haar measures on each $N_w(\A)$ and $N^w(\A)$, ($w\in W$), assigning the quotients $N_w(\Q)\backslash N_w(\A)$ and $N^w(\Q)\backslash N^w(\A)$  volume 1.  (This includes $N$  itself as a special case.)

The local groups $G(\R)$ and $G(\Q_p)$ have Iwasawa decompositions $G(\Q_v)=B(\Q_v)K_v$, where $K_p=G(\Z_p)$ for $p<\infty$ and $K_\infty$ is the maximal compact subgroup of $G(\R)$ whose Lie algebra is the complex span of all $X_\alpha-X_{-\alpha}$, ($\alpha\in \Delta$).  $B(\Q_v)$ further decomposes as $B(\Q_v)=T(\Q_v)N(\Q_v)=N(\Q_v)T(\Q_v)$.  As in Definition~\ref{def:charsont}, for any $\alpha\in\Delta$ the map $t\mapsto t^\alpha$ defines an algebraic homomorphism of $T(\Q_v)$.  Composition of this map with the valuation $|\cdot|_v$ of $\Q_v^*$ gives a real-valued character of $T(\Q_v)$. Generalizing Definition~\ref{def:charsont}, products of arbitrary complex powers of such characters,
\begin{equation}\label{TQpowerchar}
  t \ \ \mapsto \ \ |t|_v^{\l} := |t^{\a_1}|_v^{s_1}\cdots  |t^{\a_r}|_v^{s_r},
\end{equation}
 therefore assign characters of $T(\Q_v)$ to any $\l=s_1\a_1+\cdots + s_r\alpha_r\in \frak a_\C^*$, the complex span of $\Sigma$.  We parameterize such characters with a $\rho$-shift:

\begin{definition}{\bf{(The power functions $\Phi_{v,\lambda}$ and $\Phi_\lambda$)}}\label{def:powerfunction}
Let  $$\Phi_{v,\lambda}(g):=|a(g)|_v^{\l+\rho}$$ (or sometimes written just as $a(g)^{\lambda+\rho}$, when no confusion will arise), where $a(g)$ denotes a $T(\Q_v)$-factor in the Iwasawa decomposition of $g\in G(\Q_v) = N(\Q_v)T(\Q_v)K_v$. (Although $a(g)$ itself is not well-defined, $|a(g)|_v^{\l+\rho}$ is.)
\vskip 4pt
 The $\Phi_{v,\lambda}$ have the global analogs $\Phi_\lambda=\prod\limits_{v\le\infty}\Phi_{v,\l}$ (also written $|\cdot|^{\l+\rho}$), which for varying $\lambda\in\frak a_\C^*$ restrict to a continuous family of  unramified\footnote{That is, fixed under $T(\Q_v)\cap K_v$.} characters  of $T(\A)$ which are trivial on $T(\Q)$; in fact, all unramified characters of $T(\Q)\backslash T(\A)$ arise this way.

\end{definition}

 Let $P=LU$ be a standard parabolic as in Definition~\ref{def:parabolic} and let $\Sigma^L=\Sigma -\Delta_L=\Sigma\cap \Delta_U$.  If  $\lambda+\rho$ is a complex linear combination of the fundamental weights $\varpi_\alpha$, for $\alpha\in\Sigma^L$, then $\Phi_{\l}$ extends from a character of $T(\Q)\backslash T(\A)$ to a character of $L(\A)$ trivial on both $L_{der}(\A)$ and $L(\Q)$; this is because $\l+\rho$ is trivial on $H_\a$ for $\a\in\D_L$.  However, $\Phi_\l$ may still be nontrivial on $Z_L(\A)$, where $Z_L$ is the center of $L$.

Let $\frak a^*_\R\subset \frak a_\C^*$ denote the $\R$-span of $\Sigma$, i.e., the ambient real vector space of the root system $\Delta$.

\begin{definition}{\bf (The dominant chamber})
The dominant chamber $\{\mu\in \frak a^*_\R| \langle \mu,\alpha^\vee\rangle \ge 0, \alpha\in\Sigma\}$ is a fundamental domain for the action of the Weyl group $W$ on $\frak a^*_\R$.
\end{definition}

 Elements of the interior of the dominant chamber (i.e., where the defining inequality is strict for each $\alpha$) are said to be {\em strictly dominant}, and the Weyl translates of strictly dominant elements cover all of $\frak a^*_\R$ aside from a finite union of hyperplanes.  The {\em Godement range} consists of all $\lambda\in \frak a^*_\C=\frak a^*_\R+i\frak a^*_\R$ such that $\Re (\lambda-\rho)$ is strictly dominant.

\begin{definition}{\bf{(Commonly chosen additive characters of $\Q_v$)}}
Each local completion $\Q_v$ has a standard choice of additive character $\chi_v:\Q_v\rightarrow \C^*$ which is trivial on $\Z$.  For $v=\infty$ one takes $\chi_\infty(x)=e^{2\pi i x}$ and for $v=p<\infty$ one  takes $\chi_p(x)=e^{-2\pi i x_f}$, where $x_f$ is  any rational number such that $x-x_f\in\Z_p$.  These characters are well-defined and their global adelic product $\chi=\prod\limits_{v\le \infty}\chi_v$ is trivial on $\Q$.
\end{definition}

Elements  $n\in N(\A)$ can be written as products of factors $u_\alpha(x_\alpha)$, where $\alpha$ ranges over all positive roots and $x_\alpha\in\A$ depends only on $n$ and the chosen ordering of the roots in the product.  The additive homomorphism
\begin{equation}\label{def:Csum}
  C:N(\A)\rightarrow\A ,  \ \ n\mapsto \sum\limits_{\alpha\in\Sigma}x_\alpha,
\end{equation}
is well-defined independently of the choice of ordering.  In the particular example of $G=\SL_n$, this sum over the simple roots $\Sigma$ corresponds to the sum of the entries on the first superdiagonal of the matrix $n$.

\begin{definition}{\bf{(Commonly chosen nondegenerate character of $N(\Q)\backslash N(\A)$)}}\label{def:psichar}
The composition $\psi:=\chi\circ C$ is a standard choice of a character of $N(\A)$ which is trivial on $N(\Q)$ yet nondegenerate (i.e., not trivial on any of the one-parameter subgroups $\{u_\alpha(\cdot)\}$ for any $\alpha\in\Sigma$).
\end{definition}

For any $\a\in\D$, the restriction of $\psi$ to $u_\alpha(\cdot)$ is either $\chi$ (if $\alpha\in\Sigma$) or trivial (if $\alpha\notin\Sigma$). By restriction, $\psi$ also defines a nondegenerate character of $N_L=N\cap L=N\cap L_{der}$.

Suppose that $P=LU$ is a standard parabolic with standard Levi $L$, and that $\pi=\otimes_{v\le \infty}\pi_v$ is a cuspidal automorphic representation for $L$, which we assume is spherical at all places (meaning that for each $v\le \infty$ it has a $K_v$-fixed vector).  We also assume that $\pi$ is generic, meaning that it has some vector with a nontrivial Fourier coefficient corresponding to some nondegenerate character  of $N(\Q)\backslash N(\A)$.\footnote{This assumption is automatic when  $L_{der}$ is isogenous to a product of factors of Dynkin type $A_n$ \cite{PS1975,Shalika1973,Shalika1974}.}
Without loss of generality, we assume that $\pi$ transforms by a character under the center $Z_L$ of $L$. Moreover, we suppose that the central character of $\pi$ is unramified at all finite places, and therefore comes from a complex linear combination of fundamental weights $\{\varpi_{\a}|\a \in \Sigma^L\}$ as above.  The general tensor of an automorphic representation by such a character has the form $\pi\otimes |\cdot|^\lambda$, where $\lambda=\sum_{\alpha\in\Sigma^L} s_\alpha \varpi_\alpha$, ($s_\alpha\in \C$). We hence assume, as we may after twisting with a character of $Z_L(\Q)\backslash Z_L(\A)$, that  $\pi$ comes from an automorphic form for $L_{der}(\Q)\backslash L_{der}(\A)$, and transforms trivially under  $Z_L(\A)$.  Following a standard normalization in unitary induction, this last character is sometimes shifted by $\rho-\rho_L=\frac{1}{2}\sum_{\a\in\D_U}\a=\sum_{\alpha\in\Sigma^L}\varpi_\alpha$, where $\rho_L=\frac{1}{2}\sum_{\a\in\D_L}\a$.
Accordingly, for each individual element $\phi\in \pi$, define
\begin{equation}\label{philambdatwistdef}
  \phi_\l(g) \ \ := \ \ \phi(\ell(g))\ \Phi_{\l-\rho_L}(g) \quad\quad(g\in G(\A)),
\end{equation}
where $\ell(g)$ denotes an $L(\A)$ factor in the decomposition of \[ g\in G(\A)=P(\A)K_\A=L(\A)U(\A)K_\A=U(\A)L(\A)K_\A  \] and $K_\A=\prod_{v\le \infty}K_v$ (note that $\phi(\ell(g))$ is well-defined independently of the choice of $\ell(g)$). The function $\phi_\l$ is left-invariant under $P(\Q)$ by construction. For example, if $P=B$ is the Borel subgroup, then $L=T$, the representation $\pi$ is the trivial representation,  $\rho_L$ is zero, and $\phi_\l=\Phi_\l$.

\begin{definition}{\bf{(Eisenstein series for $G(\Q)\backslash G(\A)$})}\label{def:eis}
For a spherical vector $\phi\in\pi$ and $\lambda=\sum\limits_{\alpha\in\Sigma^L} s_\alpha \varpi_\alpha$,
 form the Eisenstein series
\begin{equation}\label{eisdef}
  E(g,\phi_\lambda) \ \  := \ \ \sum_{\gamma\in P(\Q)\backslash G(\Q)} \phi_\lambda(\gamma g)
\ \  = \ \ \sum_{\gamma\in P(\Q)\backslash G(\Q)} \phi(\ell(\gamma g))\Phi_{\l-\rho_L}(\g g)
\,,
\end{equation}
which  is spherical and absolutely convergent  when the complex variables $s_\alpha$, $\a\in \Sigma^L$, have sufficiently large real parts; the definition extends to all complex values of $s_\a$, $\a\in\Sigma^L$, by meromorphic continuation.
\end{definition}

\begin{definition}{\bf{(Satake/Langlands parameters)}}\label{def:satake}
Similarly to Definition~\ref{def:langlandsparameters}, any component $\pi_v$ of an automorphic representation has a Weyl orbit   of elements $\mu(\pi_v)\in {\frak a}_\C^*$ for which all vectors in $\pi_v$ have the same eigenvalues as $\Phi_{v,\mu(\pi_v)}$ under all Hecke operators (if $v<\infty$), or the ring of invariant differential operators $\mathcal D$ (if $v=\infty$).  Such $\mu(\pi_v)$ are known as {\it Satake} parameters (if $v<\infty$), or {\it Langlands} parameters (if $v=\infty$).
\end{definition}

We now consider this definition in the particular case of the Eisenstein series $E(g,\phi_\lambda)$, which has the same Satake/Langlands parameters as $\phi_\l$.  We take the convention that the Satake/Langlands parameters $\mu(\pi_v)$ for $\phi$, which  by definition are dual to the Lie algebra $\frak l$ of $L$, embed as elements of $\frak a_\C^*$ under the natural embedding of roots $\Delta_L\hookrightarrow \D$ by inclusion.  Then at the place $v$, $\phi$ has the same eigenvalues as $\Phi_{v,\mu(\pi_v)}$  under any invariant differential operator for $L$.  It is then a consequence of  Definition~\ref{def:powerfunction} that
a Satake parameter of $E(g,\phi_\lambda)$ is given by
\begin{equation}\label{bigSatake}
  \mu_v \ \ = \ \ \l \ + \ \mu(\pi_v) \ \ = \ \ \sum_{\alpha\in\Sigma^L}s_\alpha\varpi_\alpha \ + \mu(\pi_v) \qquad (v\le\infty).
\end{equation}
Pairings with coroots $\langle \mu_v,\a^\vee\rangle$   can be computed by decomposing $\a^\vee$ as a linear combination
\begin{equation}\label{corootassum}
  \a^\vee \ \ = \ \ \sum_{i\le r}c_i(\a^\vee)\,\a^\vee_i
\end{equation}
of simple coroots
and applying the defining property (\ref{fundweightdef}):~thus if $\Sigma^L$ is written as $\{\a_i|i\in S\}$ for some subset $S\subset \{1,2,\ldots,r\}$, then
\begin{equation}\label{bigSatakepairing}
  \langle\mu_v,\a^\vee\rangle  \ \ = \ \ \sum_{i\in S}\,c_i(\a^\vee)s_{\a_i} + \langle \mu(\pi_v),\a^\vee\rangle\,.
\end{equation}
These inner products appear  in  the  local factors of the Fourier coefficients of Eisenstein series.
Let
\begin{equation}\label{localzetafactors}
  \zeta_\infty(s)=\G_\R(s)=\pi^{-s/2}\Gamma(s/2), \quad \zeta_p(s)=(1-p^{-s})\i,
\end{equation}
and $c_v(s)=\f{\zeta_v(s)}{\zeta_v(s+1)}$ for $v\le \infty$.  Thus $$c(s):=\prod\limits_{v\le\infty}c_v(s)=\f{\zeta^*(s)}{\zeta^*(s+1)}=\f{\zeta^*(1-s)}{\zeta^*(1+s)},$$
where $\zeta^*$ is the completed Riemann $\zeta$-function from (\ref{zetadef}).
\begin{definition}{\bf{(Normalizing factors $\mathcal N_v$ and $\mathcal N$})}
  Define
\begin{equation}\label{Nlambdav}
\mathcal N_v(\mu)   \ \ = \ \  \prod_{\alpha\in\Delta_+}\zeta_v(\langle \mu,\alpha^\vee\rangle+1)
\end{equation}
for each $v\le \infty$
 and
 \begin{equation}\label{Nroots}
   \mathcal N((\mu_v)_{v\le\infty}) \ \  = \ \  \prod_{v\le\infty }\mathcal N_v(\mu_v)
 \end{equation}
 globally.
\end{definition}

\begin{definition}{\bf (The  Gindikin-Karpelevich integral)}
The global integral
\begin{equation}\label{GKintegral}
M_w(\lambda) \ \ :=
  \ \
  \int\limits_{N_{w\i}(\A)}   \Phi_\lambda(w\i n')\,dn' \ \ = \ \ \prod_{\alpha\in\Delta_+\cap w\i\Delta_{-}}
  c(\langle \lambda,\alpha^\vee\rangle)
\end{equation}
and its local analogs
\begin{equation}\label{GKlocal}
 M_v(\l)=\int\limits_{N_{w\i}(\Q_v)}   \Phi_{v,\lambda}(w\i n')\,dn'=c_v(\langle \lambda,\alpha^\vee\rangle)\qquad(v\in V)
\end{equation}
are known as the Gindikin-Karpelevich integrals.   These integrals and identities again initially make sense for $\lambda \in \frak a_\C^*$ having $\Re\langle \lambda,\alpha^\vee\rangle$ large\footnote{Specifically, positive for the local integrals and greater than 1 for the global integral.} for each $\alpha\in\Sigma$, and extend by meromorphic continuation to all of $\frak a_\C^*$.
\end{definition}

The Gindikin-Karpelevich integrals are the main computational tool  in Langlands' constant term formula (\ref{constterm2}) for Borel Eisenstein series via the calculation
 \begin{equation}\label{GKintegralwithPhi}
 \aligned
    \int\limits_{N_{w\i}(\A)}   \Phi_\lambda(w\i n'g)\,dn' \ \  &  = \ \
     \int\limits_{N_{w\i}(\A)}   \Phi_\lambda(w\i n'n(g)a(g))\,dn'     \\ & =  \ \
       \int\limits_{N_{w\i}(\A)}   \Phi_\lambda(w\i n'a(g))\,dn' \\ & =  \ \
 a(g)^{w\lambda+\rho} \,         \int\limits_{N_{w\i}(\A)} \Phi_\lambda(w\i n')\,dn' \\ & = \ \
         M_w(\lambda)\,\Phi_{w\lambda}(g)\qquad(g\in G(\A)),
\endaligned
 \end{equation}
where in the first equality we have factored $g=n(g)a(g)k(g)$ in Iwasawa form, in the second equality we have used the definitions of $N_{w\i}$ and $N^{w\i}$ to change variables and remove the contribution from $n(g)$,  and in the third step   changed variables $n'\rightsquigarrow a(g)n' a(g)\i$ (incurring a change of measure factor $a(g)^{\rho-w\rho}$) and used the left-transformation properties of $\Phi_{\lambda}$ under $a(g)$.

\subsection*{Example: The  $A_{n-1}$ root system and $\SL(n)$}

This explanation in fact works equally well for $\PGL(n)$, which is also a Chevalley group of rank $r=n-1$. The Cartan Lie subalgebra $\frak a$ consists of all diagonal matrices $H=\diag(h_1,\ldots,h_n)$ having $\tr(H)=h_1+\cdots+h_n=0$.
The simple roots $\{\a_1,\ldots,\a_{n-1}\}$ are linear functionals on $\frak a$, defined by the formula
$$\a_j(H)=h_j-h_{j+1},\qquad\quad (1\le j \le n-1).$$
The coroots can be identified with the roots, since $A_{n-1}$ is simply laced (i.e., $(\a,\a)$ is independent of $\a$).
The positive roots $\D_+$ are the linear functionals mapping $H$ to $h_j-h_k$, $j<k$; this last linear functional can be written as $\a_j+\a_{j+1}+\cdots+\a_{k-1}$, and so $\Delta_+$ is the set of sums $\sum_{j\le \ell<k}\a_\ell$.  Since roots are identified as coroots, this shows that the coefficient $c_i(\alpha^\vee)$ from  (\ref{corootassum}) of such a root is 1 or 0, depending on whether or not $j\le i < k$.

  The Cartan subalgebra $\frak a$ has a pairing on it given by $(X,Y):=\tr(XY)$ (i.e., the Killing form), which extends to  $\frak a^*$ by duality.  We can thus freely identify $\a_j$ with $(0,0,\ldots,0,1,-1,0,\ldots,0)$, i.e., the vector which has zeros in all positions except for a 1 in the $j$-th position, and a $-1$ in the $(j+1)$-st position.
A natural choice for the fundamental weights $\{\varpi_1,\ldots,\varpi_{n-1}\}$  in these coordinates would be
$$\varpi_1=(1,0,\ldots,0), \;\;\; \varpi_2=(1,1,0,\ldots,0), \;\;\; \varpi_3=(1,1,1,0,\ldots,0),\ldots,$$
i.e., $\varpi_j$'s first $j$ entries are 1, and the rest are zero.  However, these vectors  are technically not linear combinations of the $\a_j$ since their entries do not sum to zero.  One can remedy this either by regarding $\frak a^*$ as a quotient of $\R^n$, or instead by simply subtracting $\f jn$ from each entry of $\varpi_j$ listed above to make the entries sum to zero (this reflects the subspace-quotient duality in passing from $\frak a$ to $\frak a^*$).   It is easy to verify that the fundamental weights satisfy
 their defining property (\ref{fundweightdef}).

Standard parabolic subgroups $P=LU$ for type $A_{n-1}$ are indexed by partitions $(n_1,\ldots,n_p)$ of $n$.  The set $\Sigma^L$ of simple roots not in $L$ equals $$\{\a_{n_1},\a_{n_1+n_2},\a_{n_1+n_2+n_3},\ldots,\a_{n-n_p}\}.$$
For example, the $(1,1,1)$-parabolic of $\SL_3$ has $\Sigma^L=\{\a_1,\a_2\}$, while the $(2,1)$-parabolic of $\SL_3$ has $\Sigma^L=\{\a_2\}$.

We conclude this section by explaining how this notation specializes to the notation $|\cdot|_{\mathcal P}^s$ from Section~\ref{sec:Basicnotation}, restricting to the case of $n=3$ in order to keep the notation brief.  For the minimal parabolic $B$, we have just seen $\Sigma^L=\{\a_1,\a_2\}$ and $\lambda=s_{\a_1}\varpi_{\a_1}+s_{\a_2}\varpi_{\a_2}$, which can be written as $(s_{\a_1}+s_{\a_2},s_{\a_2},0)$ in the coordinates above.  The function $|\cdot|^\lambda$ from (\ref{TQpowerchar}) is defined via writing $\lambda$ as sum of positive roots, so we subtract the mean value $\frac{1}{3}(s_{\a_1}+2s_{\a_2})$ from each entry of this last vector in order to make the entries sum to zero, and find
$\lambda=\frac{1}{3}(2s_{\a_1}+s_{\a_2})\a_1+\frac{1}{3}(s_{\a_1}+2s_{\a_2})\a_2$.
With $t=\operatorname{diag}(t_1,t_2,t_3)$, $t_i>0$, we compute \eqref{TQpowerchar}
as
 $$\(t_1^{2s_{\a_1}+s_{\a_2}}t_2^{s_{\a_2}-s_{\a_1}}t_3^{-s_{\a_1}-2s_{\a_2}}\)^{1/3}.$$
In the case of $\SL_3$, where $t_1t_2t_3=1$, this last expression simplifies to $t_1^{s_{\a_1}+s_{\a_2}}t_2^{s_{\a_2}}$, which is what one expects from the starting formula $(s_{\a_1}+s_{\a_2},s_{\a_2},0)$ in this calculation.  It follows that the power functions in
 \eqref{Ialpha} and Definition~\ref{def:conversion}, which are themselves related by a change of coordinates of the parameters, are special cases of $\Phi_{\infty,\l}$ from Definition~\ref{def:powerfunction}.  The definitions of Eisenstein series in
Definitions~\ref{MinParEisSeries} and \ref{def:EPhi}  are then special cases of Definition~\ref{def:eis}, as can be seen by the standard conversion between classical and adelic notation (using, in particular, the fact that the inclusion $G(\Z)\hookrightarrow G(\Q)$ induces a bijection
between $(P(\Q)\cap G(\Z))\backslash G(\Z)$ and  $P(\Q)\backslash G(\Q)$).

\section{A canonical normalization of Whittaker functions}\label{sec:whit}

In this section we will define a fixed ``canonical'' normalization of spherical Whittaker functions, in order to later compare Fourier coefficients of cusp forms on a Levi subgroup to those of Eisenstein series induced from them.  This is necessary because for fixed values of $\l\in\frak a_\C^*$, Whittaker functions are only unique up to scaling.  When one varies $\l$ there is an additional choice of function to multiply them by.  This ambiguity of a nonvanishing entire function  of $\l$ is removed by our choice, which is natural from several points of view:~for example, it recovers the classical archimedean Whittaker function $2\sqrt{y}K_\nu(2\pi y)$ (asymptotic to $e^{-2\pi y}$ for large $y$) in the case of $\SL(2,\R)$.

\begin{definition} {\bf (Jacquet Integral)}
Let $\psi$ denote the standard character $\chi\circ C$ from Definition~\ref{def:Csum}, and recall the power functions $\Phi_\l$ and $\Phi_{v,\lambda}$ from Definition~\ref{def:powerfunction}.
For $\Re(\lambda-\rho)$ dominant the global Jacquet integral is defined as the absolutely convergent integral
\begin{equation}\label{Jacquet}
  \operatorname{Jac}_{\lambda}( g)  \ \  = \ \  \int_{ N(\A)}
   \Phi_\lambda(w_{\rm{long}} n g)\, \overline{\psi(n)} \, dn\,,\qquad\qquad (g\in G(\A)).
\end{equation}
which factors as the product of local Jacquet integrals
\begin{equation}\label{localJacquet}
  \operatorname{Jac}_{v,\lambda}( g)  \ \  = \ \  \int_{ N(\A)}
   \Phi_{v,\lambda}(w_{\rm{long}} n g)\, \overline{\psi(n)} \, dn\,,\qquad\qquad (g\in G(\Q_v)),
\end{equation}
 over all places $v\le\infty$.
 \end{definition}

 By construction,
 \begin{equation}\label{whittransforJac}
   \operatorname{Jac}_\lambda(ng) \ \ = \ \ \psi(n)\,   \operatorname{Jac}_\lambda(g),\qquad\qquad (n\in N(\A),\ g\in G(\A)),
 \end{equation}
 analogous to the classical situation in the remark following Definition~\ref{def:jacquetwhit}.  Since the integral (\ref{localJacquet}) is defined to be right-invariant under the maximal compact subgroup $K_v$ of $G(\Q_v)=N(\Q_v)T(\Q_v)K_v$, the value of $\operatorname{Jac}_{v,\lambda}(g)$ is determined by its values on $g\in T(\Q_v)$, where $T$ is the maximal torus from Section~\ref{sec:Chevalley2}.
 The local integrals are absolutely convergent for $\Re\l$ dominant.\footnote{The ranges of absolute convergence for the   global and local integrals (\ref{Jacquet})-(\ref{localJacquet}) are  precisely the same as those of the global and local Gindikin-Karpelevich integrals (\ref{GKintegral})-(\ref{GKlocal}).}  Both the global and local integrals have   meromorphic continuations in $\l$ to all of  ${\frak a}^*_\C$ \cite{JacquetThesis}.  Jacquet's integral is fundamental for Eisenstein series since it comes up directly in Langlands' constant term calculation (see Section~\ref{sec:eiscoeff}, in particular (\ref{constterm1})-(\ref{constterm2})).  However, it is not in general the most appropriate choice of Whittaker function from the point of view of automorphic forms, since it lacks several important symmetries and also can vanish identically at certain $\lambda$, unlike the normalized ``canonical'' Whittaker function we presently consider.  Before stating Theorem~\ref{thm:whitnorm}, we first give some examples which motivate it.

\subsection*{Example: Archimedean Whittaker functions for $\SL(2,\R)$.}\label{sec:sub:examplesl2R}

Here we parameterize $\lambda=\nu\alpha$, ($\nu\in\C$), so that $\langle \lambda+\rho,\alpha^\vee\rangle = 2\nu+1$.
The Jacquet integral at elements $g=\ttwo{\sqrt{y}}00{1/\sqrt{y}}$, $y>0$, of the connected component of the diagonal maximal torus $T(\R)$ can be explicitly calculated as
\begin{equation}\label{jacintSL2R}
 \int_\R \(\frac{y}{x^2+y^2}\)^{1/2+\nu}e^{-2\pi i x}\,dx \ \  = \ \  \frac{2 \pi ^{\nu +\frac{1}{2}} \sqrt{y}
   K_{\nu }(2 \pi  y)}{\Gamma \left(\nu
   +\frac{1}{2}\right)}\,,
\end{equation}
where $K_\nu=K_{-\nu}$ is the $K$-Bessel function.  Note that this formula shows already that the Jacquet integral (\ref{jacintSL2R}) can sometimes vanish identically in $y$, e.g., for $\nu\in -\f12+\Z_{\le 0}$ (these are points of reducibility for principal series representations that actually do not arise for spherical Maass forms).  This particular type of vanishing comes from the $\Gamma$-factor in the denominator, but the non-vanishing of the special function in the numerator is itself not obvious.

 Multiplying (\ref{jacintSL2R}) by $\G_\R(2\nu+1)=\pi^{-\nu-1/2}\Gamma(\nu+\frac 12)$ results in the simpler expression
  \begin{equation}\label{normalizedsl2Rwhit}
   2\sqrt{y} \,K_\nu(2\pi y)\,,
 \end{equation}
 which is invariant under the interchange $\nu\leftrightarrow-\nu$, and which has the {\it ``large-$y$''} asymptotics
  $e^{-2\pi y}$ independently of $\nu$ (in particular, it is never identically zero in $y$).
   For any $y>0$  (\ref{normalizedsl2Rwhit}) has an exact expression of the form $\pi ^{-\nu } \Gamma (\nu ) y^{\frac{1}{2}-\nu }f_\nu(y)+\pi
   ^{\nu } \Gamma (-\nu ) y^{\nu +\frac{1}{2}}f_{-\nu}(y)$,
   where
$$f_\nu(y) \ \ = \ \  _0F_1\left(1-\nu ;\pi ^2 y^2\right)$$ is an entire function of both $\nu$ and $y$ satisfying $f_\nu(0)=1$.  This directly recovers the {\it ``small-$y$''} asymptotics.\footnote{When $\nu$ is an integer one of $\pi ^{-\nu } \Gamma (\nu ) y^{\frac{1}{2}-\nu }f_\nu(y)$ or $\pi
   ^{\nu } \Gamma (-\nu ) y^{\nu +\frac{1}{2}}f_{-\nu}(y)$ may have a pole, but its asymptotics lie deeper than those of the other term, which in fact has a compensating pole for those asymptotics among its lower order terms;  the overall expression is entire and nonvanishing in $\nu$ despite the possible appearance of factors of $\log(y)$.}
    Despite the singularities of both terms at $\nu=0$, the overall expression nevertheless makes sense as the nonzero function $2\sqrt{y}\, K_0(2\pi y)$ there.
The simplicity and symmetry of these properties suggests the general archimedean normalization
\begin{equation}\label{Winfcanon}
  W_{\infty,\lambda}^{canon}(g) \ \ = \ \ \(\prod_{\alpha\in\Delta_+}
 \Gamma_\R( \langle \lambda,\alpha^\vee \rangle+1)\) \operatorname{Jac}_\lambda(g)\,,
\end{equation}
which is the normalization introduced below in  (\ref{Wcanon}) for $v=\infty$.

\subsection*{Example: Nonarchimedean Whittaker functions}

We refer to \cite[\S 9.3]{FGKP} for a recent account which emphasizes   Chevalley groups.
It is a consequence of (\ref{whittransforJac}) that $\operatorname{Jac}_{p,\lambda}(a)$ vanishes if $|a^{\a_i}|_p> 1$ for some $i\le r$ (i.e., $a$ is not dominant in $T(\Q_v)$).  Otherwise, the Casselman-Shalika formula   computes the Jacquet integral for dominant $a\in T(\Q_p)$ as
\begin{equation}\label{CassShal}
   \operatorname{Jac}_{p,\lambda}(a) \ \ = \ \ \prod_{\alpha\in\Delta_+}(1-p^{-\langle \lambda,\alpha^\vee \rangle-1})\cdot \sum_{w\in W} \prod_{\alpha\in\Delta_+}(1-p^{\langle w\lambda,\alpha^\vee \rangle})\i \Phi_{p,w\l}(a)\,.
\end{equation}
It can be shown that $\operatorname{Jac}_{p,\l}(e)=\prod\limits_{\alpha\in\Delta_+}(1-p^{-\langle \lambda,\alpha^\vee \rangle-1})$, so that the  normalization
\begin{equation}\label{Wpcanon}
\aligned
  W_{p,\lambda}^{canon}(g) \ \ & = \ \ \(\prod_{\alpha\in\Delta_+}(1-p^{-\langle \lambda,\alpha^\vee \rangle-1}) \)^{-1} \operatorname{Jac}_{p,\lambda}(a)
  \\
  & = \ \
  \sum_{w\in W} \prod_{\alpha\in\Delta_+}(1-p^{\langle w\lambda,\alpha^\vee \rangle})\i \Phi_{p,w\l}(a)
\endaligned
\end{equation}
manifestly satisfies
\begin{equation}\label{Wpcanonproperties}
  W_{p,w\lambda}^{canon}=W_{p,\lambda}^{canon}, \qquad\qquad W_{p,\lambda}^{canon}(e)=1.
\end{equation}
In particular, $W^{canon}_{p,\lambda}$ is never identically zero as a function of $\lambda$.  This last property in (\ref{Wpcanonproperties}) is crucial in the adelic setting, since it allows for the product of $W_p$ over  infinitely-many $p$ to be well-defined.
%

\begin{theorem}\label{thm:whitnorm} (Canonical normalization of Whittaker functions) With $\mathcal N_v(\lambda)$ as defined in (\ref{Nlambdav}), the function
\begin{equation}\label{Wcanon}
  W^{canon}_{v,\lambda}(g) \ \ := \ \ \mathcal N_v(\l) \operatorname{Jac}_{v,\lambda}(g) \qquad\qquad\qquad(g\in G(\Q_v))
\end{equation}
is entire in $\lambda$
and satisfies the functional equation $  W_{v,w\lambda}^{canon}=W_{v,\lambda}^{canon} $ for all $v\le \infty$, $\l\in\frak a^*_\C$, and $w\in W$.  Moreover, $W_{v,\l}^{canon}(e)=1$ for $v<\infty$ and
there exists an entire function  $f_\lambda$ of $\C^r$ with $f_\lambda(0)=1$ such that

\begin{equation}\label{Winfcanonident}
    W_{\infty,\lambda}^{canon}(a)   =  \sum_{w\in W}a^{  w\lambda+\rho} \( \prod_{\alpha\in\Delta_+}
 \Gamma_\R(- \langle w\lambda,\alpha^\vee \rangle) \)f_{w\lambda}(a^{\alpha_1},a^{\alpha_2},\ldots,a^{\alpha_r})
\end{equation}
for $a\in T(\R)^0$, the connected component of the group  $T(\R)$ of  invertible diagonal real matrices.
Finally, $W^{canon}_{v,\lambda}(\cdot)$ is never the zero function.
\end{theorem}

These properties suggest the terminology ``canonical normalization''.  Some of the archimedean statements are analogous to $p$-adic ones (e.g.,  (\ref{Winfcanonident}) formally has leading terms very similar to those appearing in   (\ref{CassShal})-(\ref{Wpcanon})). The normalizing factor $\mathcal N_v(\l)$
itself is also suggested by (\ref{CassShal}), (\ref{Wpcanon}), and (\ref{localzetafactors}).
Note that (\ref{Winfcanonident}) gives the asymptotics of $W_{\infty,\lambda}^{canon}(e^{-t H})$  for $H\in \frak a$ in the positive Weyl chamber and $t\rightarrow\infty$.  For example, if $\langle \lambda,\a^\vee\rangle\notin \Z$ for any $\alpha\in\Delta_+$ then

\begin{align}\label{Winfcanonasympt} & W_{\infty,\lambda}^{canon}(e^{-t H}) =\\
      &\hskip -30pt =\ \sum_{w\in W}e^{-t (w\lambda+\rho)(H)} \Bigg( \prod_{\alpha\in\Delta_+}
 \Gamma_\R( -\langle w\lambda,\alpha^\vee \rangle) + \text{lower order terms}\Bigg).\nonumber
\end{align}
These match the well-known small-$y$ asymptotics of (\ref{normalizedsl2Rwhit}).  The main new aspect here is the nonvanishing, which is shown using a deformation argument similar to the ones used in \cite{Miller2012} and \cite{HundleyMiller}.

\begin{proof}[Proof of Theorem~\ref{thm:whitnorm}]
The entirety of $\operatorname{Jac}_{v,\lambda}(g)$ along with the functional equation were proven by Jacquet in his thesis \cite{JacquetThesis}.  Note that the factor $\mathcal N_v(\lambda)$, which is the parenthetical expression in (\ref{Winfcanon}), is never zero nor singular for  $\Re\lambda$ dominant.
    Since the Jacquet integral is entire in $\lambda$, this shows that $W^{canon}_{\infty,\l}$ is holomorphic for $\Re\l$ in the dominant chamber, and hence on all of $\mathfrak a_\C^*$ by the functional equation.

Since the remaining assertions in the $p$-adic case are contained in the  example above, we assume  $v=\infty$ for the rest of the proof.
It remains to show  (\ref{Winfcanonident})  and the nonvanishing.
An expansion of the form (\ref{Winfcanonident}) was
established in greater generality by Goodman and Wallach  in \cite[(7.49)]{GoodmanWallach}.  In particular, they derive the condition that $f_\l(0)=1$.  Since the argument is short, we give a proof of this last fact here.   Recall that    $w_{\rm{long}}\rho=-\rho$.  It suffices to consider $\lambda$ with dominant real part, and compute the asymptotics of the Jacquet integral
 \begin{equation}\label{jacint1}
\aligned
  \operatorname{Jac}_{\infty,\lambda}(a) \ \ &  = \ \ \int_{N(\R)}
\Phi_{\infty,\l}(w_{\rm{long}}na)\,\overline{\psi(n)}\,dn \\
   &  = \ \ a^{2\rho}\int_{N(\R)}\Phi_{\infty,\l}(w_{\rm{long}}an)\,\overline{\psi(a n a\i)}\,dn \\
   &  = \ \ a^{w_{\rm{long}}\lambda+\rho}\int_{N(\R)}\Phi_{\infty,\l}(w_{\rm{long}}n)\,\overline{\psi(a n a\i)}\,dn \,,\quad (a\in T(\R)),\\
\endaligned
\end{equation}
for $a=e^{-tH}$, with $H$ in the dominant Weyl chamber of $\mathfrak a$ and $t\rightarrow\infty$.  This is because for such $\lambda$ and $H$, the main contribution to (\ref{Winfcanonident}) comes from the term with  $w=w_{\rm{long}}$,
\begin{equation}\label{maincontribut}
  f_{w_{\rm{long}}\l}(0) \  a^{w_{\rm{long}}\lambda+\rho}\ \prod_{\a\in\D_+}\G_\R(-\langle w_{\rm{long}}\l,\a^\vee\rangle)\,.
\end{equation}
Observing that $ana\i=e^{-tH}ne^{tH}$ tends to the identity matrix in the $t\rightarrow\infty$ limit and applying the Dominated Convergence Theorem, we see that the last integral in (\ref{jacint1}) tends to  the Gindikin-Karpelevich integral
$\int_{N(\R)}\Phi_{\infty,\l}(w_{\rm{long}}n)\,dn=M_{\infty,w_{\rm{long}}}(\lambda)$ as $t\rightarrow\infty$.

Thus the leading asymptotics of $W^{canon}_{\infty,\l}(a)$ as $t\rightarrow\infty$ are given by the equal expressions (\ref{maincontribut}) and
\begin{equation}\label{checkclaimmeansf0=1}
  a^{w_{\rm{long}}\lambda+\rho}\,{\mathcal N}_\infty(\l)\, M_{\infty,w_{\rm{long}}}(\lambda) \; = \;
a^{w_{\rm{long}}\lambda+\rho}\prod_{\a\in\D_+}
\G_\R(\langle \l,\a^\vee\rangle),
\end{equation}
by (\ref{GKlocal}).  Since $\langle w_{\rm{long}}\l,\a^\vee\rangle=\langle \l,w_{\rm{long}}\a^\vee\rangle$ and $w_{\rm{long}}$ maps $\Delta_+$ to $\Delta_{-}$, the two products over $\a\in\D_+$ are equal and hence $  f_{w_{\rm{long}}\l}(0)\equiv 1$ identically in $\l$.

Finally, the most difficult part (and the only new contribution to this proof) is to show that (\ref{Winfcanonident}) is not identically zero as a function of $H$.  This is obvious from the asymptotic expansion (\ref{Winfcanonasympt}) when all Weyl translates of $\lambda$ are distinct, but cancellation between terms can occur at   $\lambda=\lambda_0$ with dominant real part which are singular.  We will argue that the leading asymptotics of (\ref{Winfcanonident}) do not vanish for $a=e^{-tH}$, where $t$ is large and $H\in\frak a$ is chosen so that $\alpha_i(H)=(\a_i,\rho)$ for each $i\le r$, where $(\cdot,\cdot)$ is the inner product on the root system.  It follows from taking linear combinations that $\l(H)=(\lambda,\rho)$ for any $\lambda$ in the real span of $\Delta$, and furthermore in the complex span $\mathfrak a^*_{\C}$  of $\Delta$ (using the   algebraic extension of $(\cdot,\cdot)$ to $\C^r$).

Since $\Re\lambda_0$ is dominant, there is a   standard Levi subgroup $L$ for which
\begin{equation}\label{rootcharDL}
  \D_L \; = \; \{\a\in\D\mid \langle \lambda_0,\a^\vee\rangle=0\}
\end{equation}
 and the Weyl subgroup $W_L$ equals the stabilizer of $\l_0$ in $W$. For $w\in W$ it follows that
   $a^{w\lambda_0+\rho}=e^{-t (w\lambda_0+\rho)(H)}=e^{-t(w\lambda_0+\rho,\rho)}$.  Since $w_{\rm{long}}\rho=-\rho$, the most negative value of $(w\l_0,\rho)$ occurs for  $w=w_{\rm{long}}$, and consequently the leading terms in
 (\ref{Winfcanonident}) for $t\rightarrow \infty$ come from the coset $w_{\rm{long}}W_L$.  Since $f_{\mu}(0)=1$ for all $\mu$, and since each $a^{\a_i}=e^{-t\alpha_i(H)}=e^{-t(\alpha_i,\rho)}$ tends to zero as $t\rightarrow \infty$, it suffices to show that
 \begin{equation}\label{sufficesnonvanishing}
   \sum_{w\in W_L}e^{-t(w_{\rm{long}}w(\l_0+\mu),\rho)} \prod_{\a\in\D_+}\G_\R\big(-\langle w_{\rm{long}}w(\l_0+\mu),\a^\vee\rangle\big)
 \end{equation}
   is nonzero as $\mu$ approaches 0 along some sequence in $\frak a^*_\C$ (recalling that the holomorphy of this expression at $\mu=0$ has already been shown).

  We take $\mu$ of the form $\e\rho_L$, where $2\rho_L=\sum_{\a\in\D_L}\a$ and $\e$ is small.  Expression (\ref{sufficesnonvanishing}) simplifies to
    \begin{equation}\label{sufficesnonvanishing2}
  e^{t(\l_0,\rho)} \sum_{w\in W_L}e^{ \e t(w\rho_L,\rho)} \prod_{\a\in\D_+}\G_\R\big(\langle \l_0,\a^\vee\rangle+\e\langle w\rho_L,\a^\vee\rangle\big)\,,
 \end{equation}
 using the invariance of $(\cdot,\cdot)$ under $w_{\rm{long}}$ and
 again the fact that $w_{\rm{long}}$ maps $\D_+$ to $\D_{-}$.  For $\e$ small the argument of $\G_\R$ in (\ref{sufficesnonvanishing2}) is near one of the (finite set of) values  $z$ of $\langle \l_0,\a^\vee\rangle$, each of which has a nonnegative real part.  For each such value $z$, define
\begin{equation}\label{removeGammasing}
  g_z(s) \; := \; \left\{
                      \begin{array}{ll}
                        \Gamma_\R(z+s), & z\neq 0, \\
                        s\Gamma_\R(s), & z=0,
                      \end{array}
                    \right.
\end{equation}
each holomorphic and nonvanishing
for $s$ small, so that (\ref{sufficesnonvanishing2}) becomes
 \begin{equation}\label{sufficesnonvanishing3}
 e^{t(\l_0,\rho)}  \sum_{w\in W_L}e^{\e t(w\rho_L, \rho)} \;\,\frac{ \prod\limits_{\a\in\D_+}g_{\langle \l_0,\a^\vee\rangle}(\e\langle w\rho_L,\a^\vee\rangle)}{\prod\limits_{\a\in\D_+\cap \D_L} \e\langle w\rho_L,\a^\vee\rangle}\,,
 \end{equation}
where we have made use of (\ref{rootcharDL}).

  Let $\sgn(w)$ denote the determinant of  $w$ viewed as a linear transformation of the ambient vector space $\R^r$ containing the root system $\Delta$; it is equal to $(-1)^{\ell(w)}$, where $\ell(w)$ is the minimal length of $w$ when written as a product of simple Weyl reflections.  One has
\begin{equation}\label{denomandsgn}
  \prod_{\a\in\D_+\cap \D_L} \langle w\mu,\a^\vee\rangle \; = \;
 \sgn(w)\prod_{\a\in\D_+\cap \D_L} \langle \mu,\a^\vee\rangle
\end{equation}
for any $w\in W_L$ and $\mu\in \frak a^*_\C$:~this assertion reduces to the special case that $w$ is a simple Weyl reflection $w_\beta$, where it is follows from the fact that the   Weyl reflection for a simple root $\beta$ sends $\beta$ to $-\beta$, and permutes all other positive roots.  Thus the nonvanishing of (\ref{sufficesnonvanishing3}) at $\e=0$ -- where this expression is already known to be holomorphic --   follows from showing that
\begin{equation}\label{sufficesnonvanishing4}
 \left. \f{\partial^m}{\partial\e^m}\right|_{\e=0}\,\sum_{w\in W_L}e^{\e t(w\rho_L,\rho)}\sgn(w) \prod_{\a\in\D_+}g_{\langle \l_0,\a^\vee\rangle}(\e\langle w\rho_L,\a^\vee\rangle) \; \neq \; 0\,,
\end{equation}
where $m=\#(\Delta_+\cap \Delta_L)$.

 At this point we are close to a related nonvanishing statement from the Weyl character formula\footnote{This is not surprising, given the similarities with the $p$-adic formula (\ref{CassShal}), which is itself a Weyl character formula.} for the trivial representation.
 Recall that for any $\alpha_i\in\Sigma$, the pairing $(\alpha_i,2\rho)=(\alpha_i,\alpha_i)$; this is because $2\rho-\alpha_i$  is perpendicular to $\alpha_i$, being the sum of all positive roots other than $\alpha_i$, which is a set the simple reflection $w_{\alpha_i}$ permutes.
It follows that  $(w\rho_L,\rho)=(w\rho_L,\rho_L)$, since  $w\rho_L$ lies in the span of $\Delta_L$ and $(\alpha_i,\rho)=\f{(\a_i,\a_i)}{2}=(\alpha_i,\rho_L)$ for all $\a_i\in \Sigma\cap \D_L$.
 The nonvanishing of the left-hand side of (\ref{sufficesnonvanishing4}) in turn follows from showing the nonvanishing of its coefficient of $t^m$, which comes from differentiating the exponential $m$-times (but not the $g_{\langle \l_0,\a^\vee\rangle}(\cdot)$ functions at all).  Since the value of the product at $\e=0$ is independent of $w$, (\ref{sufficesnonvanishing4}) is implied by the corresponding nonvanishing statement
\begin{equation}\label{sufficesnonvanishing4.5}
 \left. \f{\partial^m}{\partial\e^m}\right|_{\e=0}\; \sum_{w\in W_L}e^{\e t(w\rho_L,\rho_L)}\sgn(w) \; \neq \; 0\,.
\end{equation}
%
However, this last nonvanishing is a simple consequence of the Weyl denominator formula for $L$,
\begin{equation}\label{weyldemonforL}
  \sum_{w\in W_L}e^{\e(w\rho_L,\rho_L)}\sgn(w) \; = \; \prod_{\a\in\D_+\cap \D_L}(e^{\e(\alpha,\rho_L)/2}-e^{-\e(\alpha,\rho_L)/2})\,,
\end{equation}
since the right-hand side is asymptotic to $\e^m\prod_{\a\in\D_+\cap \D_L}(\a,\rho_L)$ for small $\e$, and since this last product is positive.
\end{proof}

\section{Generic Fourier coefficients of cusp forms and Eisenstein series}\label{sec:eiscoeff}

Let $\pi=\bigotimes\limits_{v\le\infty}\pi_v$ denote a cuspidal automorphic representation for $G(\Q)\backslash G(\A)$, which we assume is spherical at all places.  Let $\phi\in\pi$ denote a nonzero vector which is spherical\footnote{We make this assumption so that we can give explicit results in terms of the canonically-normalized Whittaker functions from Section~\ref{sec:whit}.} at each place.
Let $\mu(\pi_v)\in {\frak a}^\star_\C$ denote a Satake parameter for $\pi_v$ ($v\le \infty$).  In (\ref{Wcanon}) we defined canonically-normalized Whittaker functions $W^{canon}_{v,\lambda}$ at each place, such that any other spherical Whittaker function at $v$ is a scalar multiple of $W^{canon}_{v,\lambda}$.  Likewise, all global Whittaker functions with parameters $\mu(\pi_v)$ ($v\le\infty$) are scalar multiples of the global product denoted
\begin{equation}\label{globalW}
  W^{\A,canon}_{(\mu(\pi_v))} \ \ := \ \ \prod_{v\le\infty}W^{canon}_{v,\mu(\pi_v)}.
\end{equation}
  In particular, with the standard choice of non-degenerate character $\psi=\chi\circ C$  of $N(\Q)\backslash N(\A)$ from Section~\ref{sec:Chevalley2},
\begin{equation}\label{phiFourier}
\aligned
\int\limits_{N(\Q)\backslash N(\A)} \phi(ng)\,\overline{\psi(n)}\,dn \; & = \; \widehat{\phi}\cdot \prod_{v\le \infty} W^{canon}_{v,\mu(\pi_v)}( g) \\
& = \; \widehat{\phi}\cdot W^{\A,canon}_{(\mu(\pi_v))}(g)    \qquad (g\in G(\A)),
\endaligned
\end{equation}
with a constant of proportionality
$\widehat{\phi}\in \C$  that will be referred to  as the (numerical) Fourier coefficient of $\phi$.
It is important to note that the definition of $\widehat{\phi}$ is premised on the nonvanishing of the canonical Whittaker functions from Theorem~\ref{thm:whitnorm}.

The above analysis is equally valid when $G$ is replaced by the semisimple subgroup $L_{der}$ of the standard Levi component $L$ of a standard parabolic subgroup.  Since the only automorphic representations we consider on $L$ are trivial under the center $Z_L$ of $L$, Whittaker functions on $L_{der}$ extend trivially over $Z_L$ to Whittaker functions on $L$.
In this case we
 add a superscript $L$ to the notation to indicate $W^{canon,L}_{v,\mu(\pi_v)}$ is a canonically-normalized Whittaker function for $L$, and have the generalization
\begin{equation}\label{phiFourierL}
  \int\limits_{N_L(\Q)\backslash N_L(\A)} \phi(n\ell)\,\overline{\psi(n)}\,dn \ \ = \ \ \widehat{\phi}\cdot \prod_{v\le \infty} W^{canon,L}_{v,\mu(\pi_v)}( \ell) \qquad (\ell\in L(\A))
\end{equation}
of
(\ref{phiFourier}).

If $\xi$ is any character of $N$, the standard unfolding method computes Fourier coefficients of the Eisenstein series $E(g,\phi_\l)$ from (\ref{eisdef}) in terms of the Bruhat decomposition  (\ref{bruhatdisjoint}) as follows (see \cite{Shahidi2010}):
\begin{multline}\label{unfolding1}
  \int\limits_{N(\Q)\backslash N(\A)} E(ng,\phi_\l)\,\overline{\xi(n)}\,dn \;
=  \int\limits_{N(\Q)\backslash N(\A)}
  \sum_{\gamma\in P(\Q)\backslash G(\Q)}
  \phi_\l(\gamma n g) \,\overline{\xi(n)}\,dn\\
  =   \int\limits_{N(\Q)\backslash N(\A)}
  \sum_{w\in W_L\backslash W}
  \sum_{\gamma'\in N_w(\Q)} \phi_\l(w\gamma' n g) \,\overline{\xi(n)}\,dn\,,\\
\end{multline}
where $\Re \lambda$ must be taken to be sufficiently dominant to ensure absolute convergence and we recall the definitions
\begin{equation}\label{recallNws}
  N_w \ = \ (w\i U_{-}w)\cap N \ \ \ \text{and} \ \ \ N^w \ = \ (w\i Pw)\cap N \quad(w\in W).
\end{equation}
The group $N_w$ is generated by root subgroups $u_\alpha(\cdot)$ in $N$ for which $wu_\alpha(\cdot)w^{-1} \subset U_{-}$,   $U_{-}$ being the opposite unipotent radical to $U$ (e.g., $LU_{-}$ is a parabolic containing $B_{-}$).  Similarly, the group $N^w$ is generated by the other root subgroups in $N$. Consequently $N(\A)$ is equal to the product of $N_w(\A)$ and $N^w(\A)$ (in either order).  Recall also that Haar measures for these groups were defined in the paragraph preceding the one containing (\ref{TQpowerchar}).

We assume $P\neq G$, for otherwise the sum defining (\ref{eisdef}) is trivial and (\ref{unfolding1}) reverts to (\ref{phiFourier}). The integration in (\ref{unfolding1}) can be rewritten as
\begin{multline}\label{unfolding2}
  \int\limits_{N(\Q)\backslash N(\A)} E(ng,\phi_\l)\,\overline{\xi(n)}\,dn\\
  = \hskip-2pt\sum_{w\in W_L\backslash W} \;
   \int\limits_{N_w(\Q)\backslash N_w(\A)}\;  \int\limits_{N^w(\Q)\backslash N^w(\A)}
  \sum_{\gamma'\in N_w(\Q)} \hskip-5pt \phi_\l\big(w\gamma' n_1 n_2 g\big) \,\overline{\xi(n_1n_2)}\,dn_2\,dn_1\\
  =  \ \ \sum_{w\in W_L\backslash W} \, \int\limits_{N_w(\A)} \;  \int\limits_{N^w(\Q)\backslash N^w(\A)}
  \phi_\l\big(wn_1n_2g\big)\,\overline{\xi(n_2)}\,dn_2\,\overline{\xi(n_1)}\,dn_1
  \\
  =  \ \ \sum_{w\in W_L\backslash W} \, \int\limits_{N_w(\A)} \;  \int\limits_{N^w(\Q)\backslash N^w(\A)}
  \phi_\l\big(wn_2n_1g\big)\,\overline{\xi(n_2)}\,dn_2\,\overline{\xi(n_1)}\,dn_1\,,
\end{multline}
where the sum over $\g'$ has been unfolded to enlarge the integration over $n_1$, and in the last step we changed the order of $n_1$ and $n_2$ inside the argument of $\phi_\l$.  Although these elements of the nilpotent group $N$  do not in general commute,      their commutators can be absorbed in a sequence of subgroups in $N$'s central series, and this change of variables is justified.  Of course these calculations are directly valid only for $\Re\l$ sufficiently dominant, but   extend by meromorphic continuation to $\l \in {\frak a}_\C^*$.

We now consider some examples and applications of (\ref{unfolding2}).

\subsection*{Constant term of the Borel Eisenstein series and the Gindikin-Karpelevich integral}

Here we specialize  $P=B$, $\phi_\l=\Phi_\l$, and $\xi$ to be trivial, so that  (\ref{unfolding2}) recovers Langlands' constant term formula as a sum over $w\in W$ of Gindikin-Karpelevich integrals.  In more detail,
\begin{multline}\label{constterm1}
  \int\limits_{N(\Q)\backslash N(\A)} E(ng,\Phi_\l)\,dn =
    \sum_{w\in  W} \int\limits_{N_w(\A)}\;  \int\limits_{N^w (\Q)\backslash N^w(\A)}
   \Phi_\lambda(wn_2n_1g)\, dn_2\,dn_1\,.
\end{multline}
Since  $wn_2w^{-1}\in N$
 and $\Phi_\lambda$ is left-invariant under $N$, the integrand is independent of $n_2$.  Since we have normalized the measure so that $ N^w(\Q)\backslash N^w(\A)$ has volume 1,
\begin{equation}\label{constterm2}
\aligned
  \int_{N(\Q)\backslash N(\A)} E(ng,\Phi_\l)\,dn \; & = \;
    \sum_{w\in   W} \int\limits_{N_w(\A)}   \Phi_\lambda(wn_1g)\,dn_1 \\ & = \;
   \sum_{w\in  W}\prod_{\alpha\in\Delta_+\cap w^{-1}\Delta_{-}}
  c(\langle \lambda,\alpha^\vee\rangle)\,\Phi_{w\lambda}(g),
\endaligned
\end{equation}
using the Gindikin-Karpelevich integral $M_{w\i}(\lambda)$ from (\ref{GKintegralwithPhi}).

\subsection*{Generic coefficients of cuspidally-induced Eisenstein series}

We now resume from (\ref{unfolding2}).  In this case we now take $\xi=\psi=\chi\circ C$ to be the standard generic character from (\ref{def:psichar}).

\begin{theorem}\label{thm:templategeneral}
With the above assumptions,
the generic Fourier coefficient of the Eisenstein series $E(\cdot,\phi_\l)$ with respect to the standard character $\psi$ from Definition~\ref{def:psichar} is
\begin{multline}\label{templateglobal}
    \int\limits_{N(\Q)\backslash N(\A)} E(ng,\phi_\l)\,\overline{\psi(n)}\,dn     =
    \widehat{\phi} \cdot
    \frac{W^{\A,canon}_{(\mu_v)}(g)}{\prod\limits_{\alpha\in\Delta_U}\prod\limits_{v\le\infty}\zeta_v(\langle \mu_v,\alpha^\vee \rangle +1)},
\end{multline}
where $\widehat\phi$ is the (numerical) Fourier coefficient from \eqref{phiFourier}, $W^{\A,canon}_{(\mu_v)}(g)$ is defined in \eqref{globalW}, and the $\mu_v$ are as defined in (\ref{bigSatake}).
\end{theorem}

Our contribution  is the archimedean theory, which allows for a precise global statement (\cite[Chapter 7]{Shahidi2010} gives a global formula, but with explicit formulas only for the nonarchimedean contributions; it is also possible to deduce Theorem~\ref{thm:templategeneral} from \cite{Shahidispaper} in terms of ratios of functionals for the automorphic representations attached to $E(\cdot,\phi_\lambda)$ and $\phi$).  The archimedean theory introduced here, in turn, rests upon (\ref{Wpcanonproperties}), which was also previously known in the nonarchimedean cases (but not in the archimedean case).  A key property of our treatment is that the canonically normalized Whittaker function $W^{\A,canon}_{(\mu_v)}$ is never identically zero.

\vskip 10pt
The denominator  in (\ref{templateglobal}) can be made explicit using (\ref{bigSatakepairing}) and (\ref{localzetafactors}).
Note that in the special case $P=B$ one has $L=T$, and $\D_U=\D_+$, so that the product in (\ref{cuspcoeff5}) is over all positive roots.
This is the heart of the template method described in the introduction:~the Fourier coefficient of the Eisenstein series equals that of the inducing cusp form, times local factors mimicking those of the Borel Eisenstein series (but notably omitting parts coming from the Levi $L$).  We give several examples in Section~\ref{sec:examplesoftemplatemethod}.

\begin{proof}
 We claim that the summands in (\ref{unfolding2}) for
\begin{equation}\label{wnotin}
 w \; \notin  \; W_L\, w_{\rm{long}}
\end{equation}
vanish identically.  Indeed, this condition is equivalent to $w w_{\rm{long}}\notin W_L$, and since $w_{\rm{long}}\D_{-}=\D_{+}$, the $w$ in (\ref{wnotin}) are precisely
those  for which $w\Delta_{+}$ is not contained in the $W_L$-orbit of $\Delta_{-}$.  In particular, for these $w$ there must exist some $\alpha\in\Delta_+$ (in fact, in $\Sigma$) such that $w\alpha$ is a root in $U$; necessarily, $u_\alpha(\cdot)$ is contained in $N^w$.  The one-parameter subgroup   $\{w\alpha(t)w^{-1}|t\in\R\}$ lies in $U(\R)$ and hence  $\phi_\l$ is independent of $t$, while the non-degenerate character $\xi$ is not.  Hence the integral over $n_2$ in (\ref{unfolding2}) vanishes unless $w\in W_Lw_{\rm{long}}$, as claimed.

Therefore,  the final line in (\ref{unfolding2}) involves only one term, for $w=w_{\rm{long}}$.  Letting $\widetilde{L}$ denote the standard Levi subgroup conjugate to $L$ by $w_{\rm{long}}$, we calculate that $N^{w_{\rm{long}}}$ is equal to $N_{\widetilde L}:=N\cap {\widetilde L}$.  Note that $N_L:=N\cap L$ is not conjugate to $N_{\widetilde L}$ by $w_{\rm{long}}$, since $w_{\rm{long}}$ conjugates $N$ to $N_{-}$; instead $N_{\widetilde L}=w_{\rm{long}}\i w_{\rm{long},L}\i N_L w_{\rm{long},L} w_{\rm{long}}$.

Inserting into (\ref{unfolding2}), we find
\begin{multline}\label{cuspcoeff1}
  \int\limits_{N(\Q)\backslash N(\A)} E(ng,\phi_\l)\,\overline{\xi(n)}\,dn
  \\  =  \int\limits_{N_{w_{\rm{long}}}(\A)} \;  \int\limits_{N_{\widetilde L}(\Q)\backslash N_{\widetilde L}(\A)}
  \phi_\l\Big(w_{\rm{long}}n_2n_1g\Big)\,\overline{\xi(n_2)}\,dn_2\,\overline{\xi(n_1)}\,dn_1 \\ =
   \int\limits_{N_{w_{\rm{long}}}(\A)}  \;  \int\limits_{N_L(\Q)\backslash N_L(\A)}
  \phi_\l\Big(n_2w_{\rm{long},L}w_{\rm{long}}n_1g\Big)\,\overline{\xi(n_2)}\,dn_2\,\overline{\xi(n_1)}\,dn_1\,,
\end{multline}
where we have used the left-invariance of $\phi_\l$ under $w_{\rm{long},L}\in W_L$.

By (\ref{philambdatwistdef}),
\begin{align}\label{philphiunravel}
   \phi_\l\Big(n_2w_{\rm{long},L}w_{\rm{long}} n_1g\Big) & =
   \phi\Big(\ell\big(n_2w_{\rm{long},L}w_{\rm{long}}  n_1g\big)\Big)\\
   &
   \hskip 36pt
   \cdot \Phi_{\l-\rho_L}\Big(n_2w_{\rm{long},L}w_{\rm{long}} n_1g\Big)
  \nonumber \\
   &
   \hskip -72pt =
   \phi\Big(\ell(n_2)\ell\big(w_{\rm{long},L}w_{\rm{long}}  n_1g\big)\Big)
   \cdot   \Phi_{\l-\rho_L}\Big(w_{\rm{long},L}w_{\rm{long}} n_1g\Big),\nonumber
\end{align}
where $\ell$ denotes an Iwasawa $L$-factor.  In particular, the argument of $\Phi_{\l-\rho_L}(\cdot)$ here is independent of $n_2$.  We now apply (\ref{phiFourierL}), which in this case reads
\begin{multline}\label{cuspcoeff2}
\int\limits_{N_L(\Q)\backslash N_L(\A)}
 \phi\Big(\ell(n_2)\ell(w_{\rm{long},L}w_{\rm{long}}  n_1g)\Big)
 \,\overline{\xi(n_2)}\,dn_2 \ \  \\ = \;\widehat{\phi} \cdot \prod_{v\le \infty} W^{canon,L}_{v,\mu(\pi_v)}\Big(\ell(w_{\rm{long},L}w_{\rm{long}} n_1 g)\Big)\,,
\end{multline}
 to   (\ref{cuspcoeff1}) and obtain

\begin{align}\label{cuspcoeff3new}
  \int\limits_{N(\Q)\backslash N(\A)}\hskip -14pt  E(ng,\phi_\l)\,\overline{\xi(n)}\,dn  & \; = \;
 \widehat{\phi}\,\cdot\hskip-14pt  \int\limits_{N_{w_{\rm{long}}}(\A)}\hskip-3pt\prod_{v\le \infty} W^{canon,L}_{v,\mu(\pi_v)}\Big(\ell\big(w_{\rm{long},L}w_{\rm{long}} n_1 g\big)\Big)\\
 &
 \hskip 43pt
 \cdot  \Phi_{\l-\rho_L}\Big(w_{\rm{long},L}w_{\rm{long}} n_1g\Big)
  \,\overline{\xi(n_1)}\,dn_1\,.
\nonumber\end{align}
It is helpful to recall that $\widehat{\phi}$ depends on the overall normalization of $\phi$.  For example, two widely used normalizations are the $L^2$ normalization and the Hecke normalization, which differ by a factor related to the special value of the adjoint $L$-function at $s=1$ (see Proposition~\ref{PropFirstCoeff} for a more precise statement).

%
%
Now we go in a reverse direction by applying (\ref{Wcanon}) (for the group $L$ instead of $G$), which gives a different integral formula for $W^{canon,L}_{v,\mu(\pi_v)}$:
\begin{align}\label{cuspcoeff4a}
  & \int\limits_{N(\Q)\backslash N(\A)}  E(ng,\phi_\l)\,\overline{\xi(n)}\,dn
    =
  \widehat{\phi} \,\cdot
  \prod_{v\le\infty}\left(
   {\mathcal N_v}^L\big(\mu(\pi_v)\big)
   \int\limits_{N_{w_{\rm{long}}}(\Q_v)}\; \int\limits_{N_L(\Q_v)}\right.\\
   &
  \hskip 80pt \cdot
    \Phi_{\mu(\pi_v)+\rho_L-\rho} \Big(w_{\rm{long},L}n_2\ell\big(  w_{\rm{long},L} w_{\rm{long}} n_1  g\big)\Big) \ \nonumber\\  &
    \hskip 68pt
    \left. \phantom{ \int\limits_{N_{w_{\rm{long}}}(\Q_v)}}
   \cdot
    \Phi_{\l-\rho_L}\Big(w_{\rm{long},L}w_{\rm{long}} n_1g\Big)
   \,\overline{\xi(n_1)} \,\overline{\xi(n_2)}\,dn_2\,dn_1
  \right),
  \nonumber
\end{align}
where the shift $\rho_L-\rho$ is needed to account for the different normalizations in Definition~\ref{def:powerfunction} for the different groups $G$ and $L$.
By construction, $\Phi_{\l-\rho_L}$ is invariant on the left under  $L_{\text{der}}(\A)$, in particular under $\ell(w_{\text{long,L}})=w_{\text{long,L}}$.  Writing $\widetilde{n}_2=w_{\text{long,L}}n_2w_{\text{long,L}}\i\in L_{\text{der}}(\A)$, the integrand can be rewritten as
\begin{equation}\label{rewriteintegrand}
  \Phi_{\mu(\pi_v)+\rho_L-\rho}\Big (\widetilde{n}_2\ell\big(   w_{\rm{long}} n_1  g\big)\Big)\,
    \Phi_{\l-\rho_L}\Big(\widetilde{n}_2 w_{\rm{long}} n_1g\Big)\,\overline{\xi(n_1n_2)}\,,
\end{equation}
where we have again appealed to the invariance of $\Phi_{\l-\rho_L}$ under  $L_{\text{der}}(\A)$.
We may also remove the map $\ell(\cdot)$ from the argument of $\Phi_{\mu(\pi_v)+\rho_L-\rho}$, because the latter function is right-invariant under all $K_v$, and so the only relevant difference between $h\in G(\A)$ and its Iwasawa factor $\ell(h)$ here is multiplication on the left by an element of $U$, a subgroup which $L$-normalizes and under which $\Phi_{\mu(\pi_v)+\rho_L-\rho}$ is left-invariant.

 In particular, (\ref{rewriteintegrand}) becomes $\Phi_{\mu_v}\Big(\widetilde{n}_2\ell\big(   w_{\rm{long}} n_1  g\big)\Big)\overline{\xi(n_1n_2)}$  using  (\ref{bigSatake}), and
  the double integral over $n_1$ and $n_2$ itself becomes a Jacquet integral for $G$:
\begin{align}\label{cuspcoeff4b}
   \int\limits_{N(\Q)\backslash N(\A)}  E(ng,\phi_\l)\,\overline{\xi(n)}\,dn \ \ &
   = \ \
  \widehat{\phi} \cdot
  \prod_{v\le\infty}\(
   {\mathcal N_v}^L(\mu(\pi_v))
    \operatorname{Jac}_{v,\mu_v}(g)\)
  \\ \nonumber & \hskip -60pt
     = \ \
  \widehat{\phi} \cdot
  \prod_{v\le\infty}\(
   {\mathcal N_v}^L(\mu(\pi_v)) \cdot \mathcal N_v(\mu_v)\i\cdot
  W^{canon}_{v,\mu_v}(g)\),
\end{align}
once more  applying (\ref{Wcanon}).
Now use (\ref{Nlambdav}) to write
\begin{equation}\label{cuspcoeff5}
  \frac{\mathcal N_v^L(\mu(\pi_v))}{\mathcal N_v(\mu_v)} \ \ = \ \ \prod_{\alpha\in\Delta_U}\zeta_v(\langle \mu_v,\alpha^\vee \rangle +1)\i
\end{equation}
as a product over the roots in $U$ (here we have used (\ref{bigSatakepairing}) to cancel the factors for $\a\in\D_L$, which have $c_i(\a^\vee)=0$ for $\a_i\in \Sigma^L$).
\end{proof}

\section{Examples of the template method}\label{sec:examplesoftemplatemethod}

Formula (\ref{cuspcoeff5}) is the theoretical basis of the template method from the introduction, as it includes the Euler factors for all positive roots aside from those coming from the Levi (see {\bf Step 3}).
When $P$ is a maximal parabolic, the product of (\ref{cuspcoeff5}) over  $v\le \infty$ can be written as follows.  The lie algebra $\frak u$ decomposes as the (finite) direct sum $\frak u={\frak u}_1\oplus {\frak u}_2\oplus\cdots$, where $\frak u_j$ is the span of all root vectors $X_\a$ for which $\langle \varpi_P,\a\rangle=j$, with $\varpi_P$ denoting the fundamental weight corresponding to the unique simple root in $\Sigma^L$.  Each $\frak u_j$ is an irreducible representation $\rho_j$ for the adjoint action of
$L$.
In this context (\ref{bigSatake}) is written as $\mu_v=s\varpi_P+\mu(\pi_v)$.
Then (\ref{cuspcoeff5}) becomes
\begin{equation}\label{cuspcoeff5formaximalparabolic}
 \frac{\mathcal N_v^L(\mu(\pi_v))}{\mathcal N_v(\mu_v)}  \ \ = \ \
  \prod_{j\ge 1} L^*(js+1,\pi,\rho_j)\i\,,
\end{equation}
the product of completed Langlands $L$-functions (see \cite[(7.1.20)]{Shahidi2010}).  By comparison, the coefficients in the constant term formula involve the ratio  $\prod\limits_{j\ge 1} \f{L^*(js,\pi,\rho_j)}{L^*(js+1,\pi,\rho_j)}$.

In this section we give some examples of (\ref{cuspcoeff5}) and (\ref{cuspcoeff5formaximalparabolic}).
All our examples are for simply-laced groups, so we adopt the simplifying  normalization that all roots have  length 2, so that roots coincide with coroots.  We frequently use the explicit notation from the beginning of Section~\ref{sec:Chevalley1} and end of Section~\ref{sec:Chevalley2} for the underlying root systems for $\SL(n)$.

\subsection*{Example: Borel Eisenstein series, with emphasis on $G=\SL(3)$}

In this case $P=B$ and $L=T$.  In particular $\Sigma^L=\Sigma$ and $\Delta_L$ is empty.  Here only the sum on the right-hand side of (\ref{bigSatake}) contributes ($\pi_v$ is trivial since it is spherical and $L$ is a torus).  Since we have arranged that roots coincide with coroots, we find
$$c_1(\alpha_1)=c_2(\alpha_2)=1, \ c_2(\alpha_1)= c_1(\alpha_2)=0,  \ c_1(\alpha_1+\alpha_2)= c_2(\alpha_1+\a_2)=1$$
in the notation of (\ref{corootassum}).  Then
$$\mu_v = s_1\varpi_1+s_2\varpi_2,$$ where we write $s_j=s_{\a_j}$ and $\varpi_j=\varpi_{\a_j}$ for shorthand.  We compute

\begin{equation}\label{borelSLn}
  \langle \mu_v,\alpha^\vee_1\rangle  =  s_1 \,,\quad
  \langle \mu_v,\alpha^\vee_2\rangle  =  s_2 \,,\quad\text{and}\ \
  \langle \mu_v,(\a_1+\alpha_2)^\vee \rangle  =  s_1+s_2 \,.
\end{equation}
In terms of matrices, $\mu_v=\diag(s_1+s_2,s_2,0)-\frac{s_1+2s_2}{3}\diag(1,1,1)$, where the second matrix is needed simply to ensure the entries sum to zero.
Thus (\ref{borelSLn})  computes the inner product of $\mu_v$ against the  positive roots, which in the notation of the example at the end of Section~\ref{sec:Chevalley2} are $(1,-1,0)$, $(0,1,-1)$, and $(1,0,-1)$, respectively.  This matches the formula in part 1) of Theorem~\ref{thm:inintro}.

\subsection*{Example: Maximal parabolic Eisenstein series for $G=\SL(3)$}

Here we take the (2,1) parabolic, so that $\D_L=\{\a_1,-\a_1\}$, and consider a cusp form $\pi$ for $\SL(2)$ with Satake parameters $\mu(\pi_v)=it_v\alpha_1$.  (The Satake parameter in terms of matrices can  be thought of as $\diag(it_v,-it_v)\in {\frak{sl}}_2(\C)$, and that matrix is $it_v$ times the diagonal matrix $\diag(1,-1)$  corresponding to the simple root of $\SL(2)$.\footnote{$\Re t_v =0$ if $\pi$ is tempered at $v$, which we expect to always be the case by the generalized Ramanujan conjectures.})     Since $\Sigma^L=\{\alpha_2\}$ consists of   the unique simple root not in the Levi, (\ref{bigSatake}) in this case  specializes to
$$\mu_v \ \ = \ \ s\varpi_2 \ + \ it_v\alpha_1,$$
where $s=s_{\a_2}$ and $\varpi_2=\varpi_{\a_2}$.
In terms of diagonal matrices, this corresponds to $\diag(s+it_v-\f{2s}{3},s-it_v-\f{2s}{3},-\f{2s}{3})$.

 In this example $\D_U=\{\a_2,\a_1+\a_2\}$ consists of the roots in the unipotent radical $U$ of the (2,1) parabolic.  In  the notation of Section~\ref{sec:Chevalley1}, the corresponding linear functionals are $\a_2(H)=h_2-h_3$ and $(\a_1+\a_2)(H)=h_1-h_3$, which correspond to the (1,3) and (2,3) entries of a matrix, respectively -- exactly the upper triangular positions not in the (2,2) block corresponding to the Levi.
We compute
\begin{equation}\label{21parabolic}
\aligned
  \langle \mu_v,\a_2^\vee\rangle \  & = \  s-it_v \quad \text{and} \\
   \langle \mu_v,(\a_1+\a_2)^\vee\rangle \  & = \  s+it_v\langle \a_1,(\a_1+\a_2)^\vee \rangle \ = \ s+it_v\,;
\endaligned
\end{equation}
again, these are the inner products of $(s+it_v,s-it_v,0)$ with $(0,1,-1)$ and $(1,0,-1)$, respectively.
Thus (\ref{cuspcoeff5}) is $L^*(s+1,\pi)\i$, verifying (\ref{cuspcoeff5formaximalparabolic}) and part 2) of Theorem~\ref{thm:inintro}.

\subsection*{Example: (2,2) parabolic for $G=\SL(4)$}

Here we set $\Sigma=\{\a_1,\a_2,\a_3\}$,
\begin{align*}
& \D_+  =\{\a_1,\  \a_2, \ \a_3, \ \a_1+\a_2, \  \a_2+\a_3, \  \a_1+\a_2+\a_3\}\,,\\  & \Delta_L=\{\pm\a_1,\pm\a_3\}, \\
&
\D_U=\{\a_2, \ \a_1+\a_2, \ \a_2+\a_3, \ \a_1+\a_2+\a_3\}.
\end{align*}
In other words, the simple roots $\a_1$ and $\a_3$ correspond to upper triangular entries in the respective $2\times 2$ blocks of the Levi, and $\D_U$ consists of the other four positive roots.

The cusp form $\pi$, without loss of generality, has the form $\pi=\pi'\boxplus\pi''$, where $\pi'$ and $\pi''$ are everywhere unramified, spherical cuspidal automorphic representations of $PGL_2$ (each is a cusp form for one of the two block matrix factors in the Levi $L$).
We thus write
$$\mu_v \; = \; s\varpi_2  + it_v'\a_1  + it_v''\a_3$$
in analogy with the maximal parabolic Eisenstein series for $\SL(3)$ just considered.
In the notation of the example at the end of Section~\ref{sec:Chevalley2}, $\mu_v$ corresponds to $(s+it'_v,s-it'_v,it''_v,-it''_v)$.
Then
\begin{equation}
\aligned
  \langle \mu_v,\a_2^\vee\rangle &  \; = \;  s \,- \, it'_v \, - \, it''_v\, , \\
  \langle \mu_v,(\a_1+\a_2)^\vee\rangle &  \; = \; s\, +\, it'_v\, -\, it''_v\,, \\
  \langle \mu_v,(\a_2+\a_3)^\vee\rangle &  \; = \;  s\, -\, it'_v\, +\, it''_v\,,\\
  \langle \mu_v,(\a_1+\a_2+\a_3)^\vee\rangle & \; = \;  s\, + \, it'_v\,+\, it''_v\,,
\endaligned
\end{equation}
which are (respectively) the inner products of $(s+it'_v,s-it'_v,it''_v,-it''_v)$ with $(0,1,-1,0)$, $(1,0,-1,0)$, $(0,1,0,-1)$, and $(1,0,0,-1)$.
Thus (\ref{cuspcoeff5formaximalparabolic}) is $L^*(s+1,\pi'\otimes\pi'')\i$.

\subsection*{Example: (2,1,1) parabolic for $\SL(4)$}

Here $\Delta_L=\{\pm\a_1\}$, and $\D_U=\{\a_2,\a_1+\a_2,\a_2+\a_3,\a_1+\a_2+\a_3,\a_3\}$, which correspond to the 5 upper triangular entries in a $4\times 4$ matrix which are not already in the upper $2\times 2$ block of the Levi.

Formula (\ref{cuspcoeff5formaximalparabolic}) does not apply since $P$ is not a maximal parabolic; instead we must return to (\ref{cuspcoeff5}).
Then with the same notation for $\mu(\pi_v)=it\a_1$ as above and writing $s_2=s_{\a_2}$ and $s_3=s_{\a_3}$, from (\ref{bigSatake}) we find
$$\mu_v \; = \;  s_2\varpi_2  +  s_3\varpi_3  +  it\a_1.$$
In terms of the notation from the end of Section~\ref{sec:Chevalley2}, this equals    $$(s_2+s_3+it-\bar{s},s_2+s_3-it-\bar{s},s_3-\bar{s},-\bar{s}), \qquad \bar{s}=\f{2s_2+3s_3}{4}. $$
We compute
\begin{equation}
\aligned
  \langle \mu_v,\a_2^\vee\rangle &  \; = \; s_2 \, - \, it\,, \\
  \langle \mu_v,(\a_1+\a_2)^\vee\rangle &  \; = \;  s_2 \, + \, it \, ,\\
  \langle \mu_v,\a_3^\vee\rangle & \; = \;  s_3\,, \\
  \langle \mu_v,(\a_2+\a_3)^\vee\rangle &  \; = \;  s_2 \, + \, s_3 \ - \ it \, ,\\
  \langle \mu_v,(\a_1+\a_2+\a_3)^\vee\rangle & \; = \;  s_2 \, + \, s_3 \ + \ it \, ,
\endaligned
\end{equation}
which are (respectively) the inner products of the previous displayed vector
with the vectors
 \[ (0,1,-1,0),\ (1,0,-1,0),\ (0,0,1,-1),\ (0,1,0,-1) \mbox{ and }(1,0,0,-1). \]
Thus (\ref{cuspcoeff5}) is $L^*(s_2+1,\pi)\i L^*(s_2+s_3+1,\pi)\i \zeta^*(s_3+1)\i$.

\subsection*{Example: (3,1) parabolic for $\SL(4)$}

Here
$$\Delta_L=\{\pm\a_1,\pm\a_2,\pm(\a_1+\a_2)\}$$
 and $\D_U=\{\a_3,\a_2+\a_3,\a_1+\a_2+\a_3\}$ (these are the three roots sending $H\mapsto h_i-h_4$, with $i=1,2,3$).
Satake parameters $\pi_v$ for $PGL_3$ are often listed in the natural $\GL_3$ coordinates as $(it_1,it_2,it_3)$ with $t_1+t_2+t_3=0$.  This convention has the property that  $\mu(\pi_v)$'s inner product  with the first simple root is $i(t_1-t_2)$, while  its inner product with the second simple root is $i(t_2-t_3)$.  Since we have normalized $\langle \a_1,\a_2\rangle = \langle \a_2,\a_2\rangle =2$ and $\langle \a_1,\a_2\rangle =-1$, we can thus write
$\mu(\pi_v) = it_1\a_1-it_3\a_2$ (since both have the same pairings against coroots).

Next, we apply (\ref{bigSatake}) to write
$$\mu_v \ \ = \ \ s\varpi_3 \ + \ it_1\a_1 \ - \ it_3\a_2\,,\qquad s=s_{\a_3},$$
which corresponds to $\diag(s+it_1,s+it_2,s+it_3,0)$ (minus the mean of the entries, to arrange for the trace to vanish).  We find
\begin{equation}\label{31parabolic}
\aligned
  \langle \mu_v,\a_3^\vee\rangle &  \; = \;   s\, + \, it_3\,, \\
  \langle \mu_v,(\a_2+\a_3)^\vee\rangle &  \; = \; s \, - \, it_1 \ - \ it_3  \; = \; s\,+\,it_2\,,\\
  \langle \mu_v,(\a_1+\a_2+\a_3)^\vee\rangle &  \; = \;  s \, + \, it_1\, ,
\endaligned
\end{equation}
which (respectively) are the inner products of $(s+it_1,s+it_2,s+it_3,0)$ with $(0,0,1,-1)$, $(0,1,0,-1)$, and $(1,0,0,-1)$.
Thus (\ref{cuspcoeff5formaximalparabolic}) is $L^*(s+1,\pi)\i$.

 {\extrarowsep=0.5mm
\begin{small}
\begin{table}[H]
\begin{center}
\textbf{FOURIER COEFFICIENTS FOR $\SL(4)$ IN TERMS OF ROOT SYSTEM PARAMETERS}
\vskip 10pt

 \begin{tabular}{|c|C{3.1cm}| C{6.3cm}|}
 \hline
 $\begin{matrix} \text{\rm Partition}\\ \text{Cusp form $\Phi$}\\
 \text{Satake pars. $\a$} \end{matrix}$& $\begin{matrix} $\text{\rm Satake parameters}$ \\ $\text{ of $E(g,\phi_\l)$}$ \end{matrix}$ & Numerical Fourier coefficient for standard generic character (without normalization factors)  \\ \hline
4 = 1+1+1+1
  &
 $(\a_1,\a_2,\a_3,\a_4)$
     &
     $\left. \;\;\;\;\;\;\bigg(\prod\limits_{1\leq j<k\leq 4} \zeta^*(1+\alpha_j-\alpha_k)\bigg)^{-1}
      \phantom{\Bigg(xx\Bigg)_{\int_{P^2}}} \right.
      $  \\ \hline
  $\begin{matrix} 4 = 2+1+1\\
  \phi \; \text{on $\GL(2)$}\\
it\a_1= (it,-it)   \end{matrix}$
  &
  \hskip -2pt
  $ s_2\varpi_2+s_3\varpi_3+it\a_1$
    &
    $ \begin{matrix}
&\hskip-40pt
  \big(\zeta^*(s_3+1)  \, L^*(s_2+1, \phi) \\ & \hskip 50pt \cdot \ L^*(1+s_2+s_3, \phi)\big)^{-1}
\end{matrix}
$
    \\ \hline
 $\begin{matrix} \phantom{.}\\ 4= 2+2\\
 \Phi = (\phi_1, \, \phi_2)\\   \text{on}\\
  \GL(2)\times \GL(2)\\
 it'\a_1, it''\a_3 \\
 \phantom{.}  \end{matrix}$
  &
   $s\varpi_2+it'\a_1+it''\a_3$
   &
   \hskip-7pt
   $
    L^*\big(s+1,\phi_1\times\phi_2\big)^{-1}
   \ $
    \\ \hline
 $\begin{matrix} 4 = 3+1 \\
 \phi \;\text{on $\GL(3)$} \\
it_1\a_1-it_3\a_2 \end{matrix}$
 &
 $s\varpi_3+it_1\a_1-it_3\a_2$
  &
  $ L^*(s+1, \phi)^{-1} $
   \\ \hline
\end{tabular}

\end{center}
\end{table}
\end{small}
}

\begin{remark}
This table is parallel to the earlier one in Section~\ref{sec:sketch}, except that we have omitted the normalizing factor $\widehat\phi$ (which is the source of the special values of Adjoint $L$-functions in the earlier table), as well as the Hecke eigenvalues.
\end{remark}

\subsection*{Example: $E_7$ parabolic for $E_8$}

  Here   $\frak u=\frak u_1\oplus \frak u_2$, where $\frak u_1$ is 56-dimensional and $\frak u_2$ is one-dimensional. Let $\pi$ be a globally spherical and unramified cuspidal automorphic representation of $E_7^{sc}$ (the simply-connected Chevalley group of type $E_7$).
  Then (\ref{cuspcoeff5formaximalparabolic}) equals  $L^*(s+1,\pi,{\bf 56})\i\zeta^*(2s+1)\i$, where $\bf 56$ denotes the 56-dimensional representation of $E_7^{sc}\subset Sp(56)$.

\subsection*{Example: $D_7$ parabolic for $E_8$}

Here $\frak u$ is 78 dimensional, and $\frak u=\frak u_1\oplus\frak u_2$ decomposes as the direct sum of the 64-dimensional spin representation of $Spin(7,7)$, and its 14-dimensional representation ``$Stan$'' (which factors through $SO(7,7))$.  Thus  for $\pi$ a globally spherical and unramified cuspidal automorphic representation  of $Spin(7,7)$, (\ref{cuspcoeff5formaximalparabolic}) equals $L^*(s+1,\pi,Spin)\i L(2s+1,\pi,Stan)\i$.

\vspace{0.5cm} \indent {\it
A\,c\,k\,n\,o\,w\,l\,e\,d\,g\,m\,e\,n\,t\,s.\;} The authors would like to thank Herv\'e Jacquet, Freydoon Shahidi, and Nolan Wallach for many helpful discussions.

\bibliographystyle{amsalpha}

\bibliography{biblio}

\bigskip
\bigskip
\begin{minipage}[t]{10cm}
\begin{flushleft}
\small{
\textsc{Dorian Goldfeld}
\\*Department of Mathematics,
\\*Columbia University,
\\*2990 Broadway
\\* New York, NY 10027, USA
\\*e-mail: dg15@columbia.edu
\\[0.4cm]
\textsc{Stephen D.~Miller}
\\*Department of Mathematics,
\\*Rutgers, The State University of New Jersey,
\\*110 Frelinghuysen Rd
\\*Piscataway, NJ 08854-8019, USA
\\*e-mail: miller@math.rutgers.edu
\\[0.4cm]
\textsc{Michael Woodbury}
\\*Department of Mathematics,
\\*Columbia University,
\\*2990 Broadway
\\*New York, NY 10027, USA
\\*e-mail: woodbury@math.columbia.edu
}
\end{flushleft}
\end{minipage}


\end{document}